\documentclass[fleqn, 11pt]{article}
\usepackage{amsmath, amssymb, amsthm, graphicx}
\usepackage{hyperref}
\usepackage{cite}
\theoremstyle{definition}
\newtheorem{theorem}{Theorem}[section]
\newtheorem{criterion}{Criterion}[section]

\newtheorem{lemma}[theorem]{Lemma}

\newtheorem{Remark}[theorem]{Remark}
\newenvironment{remark}{\begin{Remark}\rm}{\end{Remark}}

\newtheorem{Example}[theorem]{Example}
\newtheorem{Question}[theorem]{Question}

\numberwithin{equation}{section}

\begin{document}

\title{Classification of positive elliptic-elliptic rotopulsators on Clifford tori}

\author{Pieter Tibboel \\
Bernoulli Institute for Mathematics\\
University of Groningen\\
P.M.J.Tibboel@rug.nl\\
https://orcid.org/0000-0001-6292-897X}




\maketitle
\begin{abstract}
  We prove that positive elliptic-elliptic rotopulsator solutions of the $n$-body problem in spaces of constant Gaussian curvature that move on Clifford tori of nonconstant size either lie on great circles, or project onto regular polygons. We additionally prove for the case that the configurations project onto regular polygons that all masses are equal and show that all these different types of positive elliptic-elliptic rotopulsator exist.
\end{abstract}
{\bf Keywords:} $n$-body problems; curved $n$-body problem; ordinary differential equations; dynamical systems; celestial mechanics
\section{Introduction}
  By the $n$-body problem in spaces of constant Gaussian curvature, or curved $n$-body problem for short, we mean the problem of finding the dynamics of point masses that move on a sphere, or a hyperboloid. Specifically, if $\sigma=\pm 1$, then we mean the problem of finding point masses \begin{align*}q_{1}, \ldots, q_{n}\in\mathbb{M}_{\sigma}^{3}=\{(x_{1},x_{2},x_{3},x_{4})^{T}\in\mathbb{R}^{4}\mid x_{1}^{2}+x_{2}^{2}+x_{3}^{2}+\sigma x_{4}^{2}=\sigma\}\end{align*} with respective masses $m_{1}>0$, \ldots, $m_{n}>0$, solving the system of differential equations
  \begin{align}\label{EquationsOfMotion Curved}
   \ddot{q}_{i}=\sum\limits_{\substack{j=1\\j\neq i}}^{n}\frac{m_{j}(q_{j}-\sigma(q_{i}\odot q_{j})q_{i})}{(\sigma -\sigma(q_{i}\odot q_{j})^{2})^{\frac{3}{2}}}-\sigma(\dot{q}_{i}\odot\dot{q}_{i})q_{i},\quad i\in\{1, \ldots,\quad n\},
  \end{align}
  where for $x$, $y\in\mathbb{M}_{\sigma}^{3}$  the product $\cdot\odot\cdot$ is defined as
  \begin{align*}
    x\odot y=x_{1}y_{1}+x_{2}y_{2}+x_{3}y_{3}+\sigma x_{4}y_{4}.
  \end{align*}
  The study of the curved $n$-body problem has applications to, for example, geometric mechanics, Lie groups and algebras, non-Euclidean and differential geometry and stability theory, the theory of polytopes and topology (see for example \cite{D6}) and for $n=2$ goes back as far as the 1830s (see \cite{BM,BMK,CRS,DK,DPS1,DPS2,DPS3} and \cite{KH} for a historical overview and recent results). However, the first paper giving an explicit $n$-body problem in spaces of constant Gaussian curvature for general $n\geq 2$ was published in 2008 by Diacu, P\'erez-Chavela and Santoprete (see \cite{DPS1,DPS2,DPS3}). This breakthrough then gave rise to further results for the $n\geq 2$ case in \cite{DeDZ}, \cite{D1,D2,D3,D4,D5,D6,D7,DK,DPo,DT,PS,PST,SZ,T,T2,T3,T4,T5,T6,T7,T9,T10,Z1,YZ,ZZ} and the references therein.

  Rotopulsators are solutions to (\ref{EquationsOfMotion Curved}) consisting of any orbit induced by a (possibly hyperbolic) rotation, but otherwise impose few restrictions on the position vectors of the point masses. They were introduced in \cite{DK} by Diacu and Kordlou as generalisations of homographic orbit solutions, i.e. solutions of which the shape of the configuration stays the same over time, but the size may change, i.e. the ratio of any two distances between point masses is constant.
  Specifically, let
  \begin{align*}
    R(x)=\begin{pmatrix}
      \cos{x} & -\sin{x} \\
      \sin{x} & \cos{x}
    \end{pmatrix}\quad\textrm{and}\quad S(x)=\begin{pmatrix}
      \cosh{x} & \sinh{x} \\
      \sinh{x} & \cosh{x}
    \end{pmatrix}.
  \end{align*}
  If we write $q_{i}=(q_{i1},q_{i2},q_{i3},q_{i4})^T$, $i\in\{1, \ldots, n\}$, then we call any solution $q_{1}, \ldots,$ $q_{n}$  of (\ref{EquationsOfMotion Curved}) a \textit{rotopulsator}, if there exist nonnegative scalar functions $r_{i}$, $\rho_{i}$ for which $r_{i}^{2}+\sigma\rho_{i}^{2}=\sigma$, scalar functions $\theta$, $\phi$ and constants $\alpha_{i}$, $\beta_{i}\in\mathbb{R}$ such that
  \begin{align}
    &\begin{pmatrix}
      q_{i1}\\ q_{i2}
    \end{pmatrix}=r_{i}R(\theta+\alpha_{i})\begin{pmatrix}
      1 \\ 0
    \end{pmatrix},\quad\begin{pmatrix}
      q_{i3}\\ q_{i4}
    \end{pmatrix}=\rho_{i}R(\phi+\beta_{i})\begin{pmatrix}
      1 \\ 0
    \end{pmatrix},\quad\textrm{if}\quad\sigma=1,\\
    &\quad\textrm{or}\quad\begin{pmatrix}
      q_{i3}\\ q_{i4}
    \end{pmatrix}=\rho_{i}S(\phi+\beta_{i})\begin{pmatrix}
      0 \\ 1
    \end{pmatrix}\quad\textrm{if}\quad\sigma=-1.\label{Rotopulsator identities}
  \end{align}
  If $\sigma=1$ and the $\alpha_{i}$ are not necessarily all the same value and the same holds true for the $\beta_{i}$, then we call such a solution a \textit{positive elliptic-elliptic rotopulsator}.

  A potentially useful application of rotopulsators is that they may give information about the geometry of the universe: For example, Diacu, P\'erez-Chavela and Santoprete (see \cite{DPS1}, \cite{DPS2}) showed that a relative equilibrium (a rotopulsator induced by a rotation only, i.e. the shape and size of the configuration do not change over time) with an equilateral triangle configuration has to have equal masses in spaces of constant Gaussian curvature. As the Sun, Jupiter and the Trojan asteroids (approximately) have an equilateral triangle configuration and their masses are unequal, this means that the space occupied by their configuration is likely flat. \newline

  Secondly, with continued investigations into whether solutions to the curved $n$-body problem can be related to the classical $n$-body problem (see for example \cite{BGPR,D7}), rotopulsators seem to be good candidates for solutions to the curved $n$-body problem that are somehow related to homographic orbit solutions and relative equilibria solutions to the Newtonian $n$-body problem. It may be possible to use results for rotopulsators (see for example \cite{T,T2,T6}) and using such a relationship to shed a new light on various open problems related to polygonal relative equilibria (see \cite{ACS} for such open problems). \newline

  Finally, because of their very general setup, investigating rotopulsators has a strong potential to give rise to new dynamics of orbits in curved space. \newline

  The most commonly investigated class of rotopulsators in the literature is the class of rotopulsators for which $r_{i}$, $\rho_{i}$ are independent of $i$, $i\in\{1,\ldots,n\}$ (see for example \cite{DK,DT,T,T2,T6}). For general $n$ it was proven in \cite{T6} for this case that hyperbolic rotopulsators, i.e. if $\sigma=-1$, the configuration of the point masses has to be a regular polygon and all masses have to be equal. In \cite{T2} it was proven for the same class, but now for rotopulsators on $\mathbb{S}^{3}$, i.e. $\sigma=1$, that if the $\beta_{i}$ are all equal, the configuration of the point masses has to be a regular polygon as well and in \cite{T6} it was proven that then the masses have to be equal. The only remaining case to be conclusively investigated of this subclass of rotopulsators is the case that the solution to (\ref{EquationsOfMotion Curved}) is a positive elliptic-elliptic rotopulsator and $r_{i}=:r$ and $\rho_{i}=:\rho$ are independent of $i$ and not constant, i.e. the point masses move along a Clifford torus
  \begin{align*}\{(x_{1},x_{2},x_{3},x_{4})\in\mathbb{R}^{4}\mid x_{1}^{2}+x_{2}^{2}=r^{2},\quad x_{3}^{2}+x_{4}^{2}=\rho^{2}\}, \end{align*}
  or equivalently, the vectors
  \begin{align*}
    \begin{pmatrix}
      q_{i1}\\ q_{i2}\end{pmatrix}\quad\textrm{and}\quad\begin{pmatrix}
      q_{i3}\\ q_{i4}
    \end{pmatrix},\quad i\in\{1,\ldots,n\}
  \end{align*}
  correspond to vertices of polygons inscribed in circles of radius $r$ and $\rho$ respectively.
  \begin{remark} The most general results for $n>3$ on classifying solutions of this type for the case that the $\beta_{i}$ are not necessarily all equal was for $n=4$ in \cite{DT}, but under the assumption that the configurations are rectangular. In this paper we classify all rotopulsators of this type for general $n$, without any assumptions on the configurations. \end{remark}
  \begin{remark}
    Throughout this paper we will use the following terminology when it comes to vertices of polygons: Let $p\in\mathbb{N}$ and let $v_{1}$, \ldots, $v_{p}\in\mathbb{R}^{2}$ be not necessarily distinct vectors. If there exist distinct vectors $\{v_{a_{j}}\}_{j=1}^{k}\subset\{v_{i}\}_{i=1}^{p}$ such that for all $i\notin\{a_{j}\}_{j=1}^{k}$ we have that there is a $\widehat{j}\in\{a_{j}\}_{j=1}^{k}$ such that $v_{i}=v_{\widehat{j}}$, then we say that the vectors $v_{1}$, \ldots, $v_{p}$ represent the vertices of a polygon with $k$ vertices.
  \end{remark}
  In this paper, we will prove that all positive elliptic-elliptic rotopulsators for which $r_{i}$ and $\rho_{i}$ are not constant and independent of $i$ either have configurations of point masses that are regular polygons with all masses equal, project onto regular polygons and have all masses equal, or have configurations that lie on great circles. Specifically:
  \begin{theorem}\label{Main Theorem 1}
    Let $i\in\{1,\ldots,\quad n\}$. Let $q_{1}$, \ldots, $q_{n}$ be a positive elliptic-elliptic rotopulsator solution of (\ref{EquationsOfMotion Curved}) with $r_{i}$ and $\rho_{i}$ independent of $i$ and not constant. Let $V_{i}=\{j\in\{1,\ldots,n\}\mid\beta_{j}=\beta_{i}\}$ and $W_{i}=\{j\in\{1,\ldots,n\}\mid\alpha_{j}=\alpha_{i}\}$. If there is an $i\in\{1,\ldots,n\}$ such that $|V_{i}|>1$, then $|V_{i}|=|V_{k}|$ for all $i$, $k\in\{1,\ldots,n\}$, $|W_{i}|=|W_{k}|$ for all $i$, $k\in\{1,\ldots,n\}$, the vectors $(q_{j1},q_{j2})^{T}$, $j\in\{1,\ldots,n\}$ represent the vertices of a regular polygon with $|V_{i}|$ vertices, the vectors $(q_{j3},q_{j4})^{T}$, $j\in\{1,\ldots,n\}$ represent the vertices of a regular polygon with $|W_{i}|$ vertices and all masses are equal.
  \end{theorem}
  \begin{theorem}\label{Main Theorem 2}
    There exist positive elliptic-elliptic rotopulsator solutions $q_{1}$, \ldots, $q_{n}$ of (\ref{EquationsOfMotion Curved}) with $r_{i}$ and $\rho_{i}$ independent of $i$ and not constant and $\|q_{i}(t)-q_{j}(t)\|$ constant for all $i$, $j\in\{1,\ldots,n\}$, where $\|\cdot\|$ represents the Euclidean norm on $\mathbb{R}^{4}$. For all these solutions we have that $\alpha_{i}=\beta_{i}$ for all $i\in\{1,\ldots,n\}$ and any such solution has a (not necessarily regular) polygonal configuration, where all the point masses lie on a great circle.
  \end{theorem}
  \begin{theorem}\label{Main Theorem 1a}
     Let $i\in\{1,\ldots,n\}$. Let $q_{1},\ldots,q_{n}$ be a positive elliptic-elliptic rotopulsator solution of (\ref{EquationsOfMotion Curved}) with $r_{i}$ and $\rho_{i}$ independent of $i$ and not constant.
     If the $\alpha_{i}$ are all distinct and the $\beta_{i}$ are all distinct, then the configurations of the $(q_{i1},q_{i2})^{T}$ and of the $(q_{i3},q_{i4})^{T}$ are regular polygons with $n$ vertices and equal masses, or the $q_{i}$ represent vertices of a (not necessarily regular) polygon inscribed in a great circle.
  \end{theorem}
  \begin{theorem}\label{Main Theorem 3}
    There exist positive elliptic-elliptic rotopulsator solutions $q_{1}$, \ldots, $q_{n}$ of (\ref{EquationsOfMotion Curved}) for which $r_{i}$ and $\rho_{i}$ are independent of $i$, the $(q_{i1},q_{i2})^{T}$, $i\in\{1,\ldots,n\}$ and the $(q_{i3},q_{i4})^{T}$, $i\in\{1,\ldots,n\}$ represent vertices of a regular polygon and the masses are all equal. Such solutions exist both for the case that $r_{i}$ is constant and independent of $i$ and for the case that $r_{i}$ is nonconstant and independent of $i$.
  \end{theorem}
  \begin{remark}
     Note that by Theorems~\ref{Main Theorem 1}--\ref{Main Theorem 3} and Theorem~1.3 and Theorem~1.4 in \cite{T6}, we have proven that all rotopulsators for which $r_{i}$ and $\rho_{i}$ are not constant and independent of $i$ have configurations of point masses that are either regular polygons and have all masses equal, project onto regular polygons and have all masses equal, or lie on great circles.
  \end{remark}
  We will now first formulate a criterion and lemmas needed to prove Theorems~\ref{Main Theorem 1}--\ref{Main Theorem 3} in section~\ref{Section Background Theory}, after which we will prove Theorems~\ref{Main Theorem 1}, \ref{Main Theorem 2}, \ref{Main Theorem 1a} and \ref{Main Theorem 3} in sections~\ref{Section proof of main theorem 1}, \ref{Section proof of main theorem 2}, \ref{Section proof of main theorem 1a} and \ref{Section proof of main theorem 3} respectively.
  \section{Background theory}\label{Section Background Theory}
  In this section we will formulate the background theory we need to prove our main results and the main difficulties we will need to overcome to actually prove those results. We will use the notation introduced in the previous section. Additionally, we will adopt the notation $\langle\cdot,\cdot\rangle$ for the Euclidean inner product. The following lemma and criterion form the backbone of our proofs of our main results and were proven for general $r_{i}$ and $\rho_{i}$, $i\in\{1,\ldots,n\}$ in \cite{DK}:
  \begin{lemma}\label{Lemma 1}
    Let $q_{1},\textrm{ }\ldots,\textrm{ }q_{n}$ constitute a positive elliptic-elliptic rotopulsator with $r_{i}=r$ and $\rho_{i}=\rho$ independent of $i$. Then $2\dot{r}\dot{\theta}+r\ddot{\theta}=0$ and $2\dot{\rho}\dot{\phi}+\rho\ddot{\phi}=0$, i.e. there exist constants $C_{1}$, $C_{2}$ such that $r^{2}\dot{\theta}=C_{1}$ and $\rho^{2}\dot{\phi}=C_{2}$.
  \end{lemma}
  \begin{remark}
    Lemma~\ref{Lemma 1} was proven for general $r_{i}$ and $\rho_{i}$ in \cite{DK} (see Criterion~2, equation (34)). For the case that $r_{i}$ and $\rho_{i}$ are independent of $i$, the result immediately gives Lemma~\ref{Lemma 1}.
  \end{remark}
%
  \begin{criterion}\label{Criterion}
    Let $q_{1},\textrm{ }\ldots,\textrm{ }q_{n}$ be a positive elliptic-elliptic rotopulsator with $r_{i}=r$ and $\rho_{i}=\rho$ independent of $i$. Then
    \begin{align}
      0 &=\sum\limits_{\substack{j=1\\j\neq i}}^{n}\frac{m_{j}\sin{(\alpha_{j}-\alpha_{i})}}{\left(1-\left(\cos{(\beta_{j}-\beta_{i})}+r^{2}\left(\cos{(\alpha_{j}-\alpha_{i})}-\cos{(\beta_{j}-\beta_{i})}\right)\right)^{2}\right)^{\frac{3}{2}}}, \label{Crit1}\\
      0 &=\sum\limits_{\substack{j=1\\j\neq i}}^{n}\frac{m_{j}\sin{(\beta_{j}-\beta_{i})}}{\left(1-\left(\cos{(\beta_{j}-\beta_{i})}+r^{2}\left(\cos{(\alpha_{j}-\alpha_{i})}-\cos{(\beta_{j}-\beta_{i})}\right)\right)^{2}\right)^{\frac{3}{2}}}\quad\textrm{and}\label{Crit2}\\
      \delta &=r\rho^{2}\sum\limits_{\substack{j=1\\j\neq i}}^{n}\frac{m_{j}(\cos{(\alpha_{j}-\alpha_{i})}-\cos{(\beta_{j}-\beta_{i})})}{\left(1-\left(\cos{(\beta_{j}-\beta_{i})}+r^{2}\left(\cos{(\alpha_{j}-\alpha_{i})}-\cos{(\beta_{j}-\beta_{i})}\right)\right)^{2}\right)^{\frac{3}{2}}},\label{Crit3}
    \end{align}
    where $\delta=\ddot{r}+r\rho^{2}((\dot{\phi})^{2}-(\dot{\theta})^{2})+r\left(\frac{\dot{r}}{\rho}\right)^{2}$.
  \end{criterion}
  \begin{remark}
    Criterion~\ref{Criterion} was proven for general $r_{i}$, $\rho_{i}$ in \cite{DK} (see Criterion~2). For the case that $r_{i}$ and $\rho_{i}$ are independent of $i$, the Criterion~2 in \cite{DK} immediately gives (\ref{Crit1})--(\ref{Crit3}). The idea behind the proof comes down to inserting our expressions for positive elliptic-elliptic rotopulsators (see \ref{Rotopulsator identities}) into (\ref{EquationsOfMotion Curved}), multiplying both sides of the resulting equations for the first and second coordinates of the $q_{i}$ from the left with $R(\theta+\alpha_{i})^{-1}$ and the resulting equations for the third and fourth coordinates of the $q_{i}$ from the left with $R(\phi+\beta_{i})^{-1}$ then gives the right-hand sides of (\ref{Crit1})--(\ref{Crit3}), whereas using Lemma~\ref{Lemma 1} then gives the left-hand sides of (\ref{Crit1}) and  (\ref{Crit2}).
  \end{remark}
  With Criterion~\ref{Criterion} in place, we can now formulate the basic idea behind the proofs of our main results: To prove Theorems~\ref{Main Theorem 1} and \ref{Main Theorem 1a}, 
  ideally we would like to use that the terms on the righthand sides of the identities of Criterion~\ref{Criterion} are linearly independent given certain conditions on the $\alpha_{i}$ and $\beta_{i}$, $i\in\{1,\ldots,n\}$, if $r_{i}=r$ and $\rho_{i}=\rho$ are independent of $i$ and not constant. However, for such a linear independence argument to work, we need to address the possibility of the  $\cos{(\alpha_{j}-\alpha_{i})}-\cos{(\beta_{j}-\beta_{i})}$ being zero for certain $i$, $j$, in which case our calculations may become much more complex. To this end, we will formulate the next lemma, which states the exact  conditions for the terms of the sums on the right-hand side of (\ref{Crit1})--(\ref{Crit3}) to be linearly independent and a second lemma, which deals exclusively with the possibility of $\beta_{j}-\beta_{i}=0\textrm{ }(mod\textrm{ }2\pi)$ for certain $i$, $j$:
  \begin{lemma}\label{Lemma 2}
    Let $q_{1},\textrm{ }\ldots,\textrm{ }q_{n}$ be a positive elliptic-elliptic rotopulsator solution of (\ref{EquationsOfMotion Curved}) with, $r_{i}=r$ and $\rho_{i}=\rho$ independent of $i$ and not constant. Let $A_{j_{1}i_{1}}$ and $A_{j_{2}i_{2}}$ be nonzero constants, $i_{1}$, $i_{2}$, $j_{1}$, $j_{2}\in\{1,\ldots,n\}$, $j_{1}\neq i_{1}$ and $j_{2}\neq i_{2}$. Then for terms
    \begin{align}\label{TermA1}
      \frac{A_{j_{1}i_{1}}}{\left(1-\left(\cos{(\beta_{j_{1}}-\beta_{i_{1}})}+r^{2}\left(\cos{(\alpha_{j_{1}}-\alpha_{i_{1}})}-\cos{(\beta_{j_{1}}-\beta_{i_{1}})}\right)\right)^{2}\right)^{\frac{3}{2}}},
    \end{align}
    $\cos{(\alpha_{j_{1}}-\alpha_{i_{1}})}-\cos{(\beta_{j_{1}}-\beta_{i_{1}})}\neq 0$, to cancel out against terms
    \begin{align}
      \frac{A_{j_{2}i_{2}}}{\left(1-\left(\cos{(\beta_{j_{2}}-\beta_{i_{2}})}+r^{2}\left(\cos{(\alpha_{j_{2}}-\alpha_{i_{2}})}-\cos{(\beta_{j_{2}}-\beta_{i_{2}})}\right)\right)^{2}\right)^{\frac{3}{2}}}\label{TermA2}
    \end{align}
    in the sums in the right-hand sides of (\ref{Crit1})--(\ref{Crit3}), we need that either  \begin{align*}\cos{(\alpha_{j_{1}}-\alpha_{i_{1}})}=\cos{(\alpha_{j_{2}}-\alpha_{i_{2}})}\quad\textrm{and}\quad\cos{(\beta_{j_{1}}-\beta_{i_{1}})}=\cos{(\beta_{j_{2}}-\beta_{i_{2}})}, \end{align*}
    or
    \begin{align*}\cos{(\alpha_{j_{1}}-\alpha_{i_{1}})}=-\cos{(\alpha_{j_{2}}-\alpha_{i_{2}})}\quad\textrm{and}\quad\cos{(\beta_{j_{1}}-\beta_{i_{1}})}=-\cos{(\beta_{j_{2}}-\beta_{i_{2}})} \end{align*}
    \end{lemma}
  \begin{proof}
    If $\cos{(\alpha_{j_{1}}-\alpha_{i_{1}})}\neq\cos{(\beta_{j_{1}}-\beta_{i_{1}})}$, then terms (\ref{TermA1})
    are linearly independent for different values of $i_{1}$, $j_{1}$ if and only if the roots of the polynomials $1-(\cos{(\beta_{j_{1}}-\beta_{i})}+x(\cos{(\alpha_{j_{1}}-\alpha_{i})}-\cos{(\beta_{j_{1}}-\beta_{i})}))^{2}$ differ for different values of $j_{1}$. Because the roots of \begin{align*}(1-(\cos{(\beta_{j_{1}}-\beta_{i_{1}})}+x(\cos{(\alpha_{j_{1}}-\alpha_{i_{1}})}-\cos{(\beta_{j_{1}}-\beta_{i_{1}})}))^{2})\end{align*} are \begin{align*}\frac{1-\cos{(\beta_{j_{1}}-\beta_{i_{1}})}}{\cos{(\alpha_{j_{1}}-\alpha_{i_{1}})}-\cos{(\beta_{j_{1}}-\beta_{i_{1}})}}\quad\textrm{and}\quad -\frac{1+\cos{(\beta_{j_{1}}-\beta_{i_{1}})}}{\cos{(\alpha_{j_{1}}-\alpha_{i_{1}})}-\cos{(\beta_{j_{1}}-\beta_{i_{1}})}}\end{align*} if $\cos{(\alpha_{j_{1}}-\alpha_{i_{1}})}\neq\cos{(\beta_{j_{1}}-\beta_{i_{1}})}$, we have that terms (\ref{TermA1})
    can cancel out against terms
    (\ref{TermA2})
    if and only if
    \begin{align*}\frac{1-\cos{(\beta_{j_{1}}-\beta_{i_{1}})}}{\cos{(\alpha_{j_{1}}-\alpha_{i_{1}})}-\cos{(\beta_{j_{1}}-\beta_{i_{1}})}}=\frac{1-\cos{(\beta_{j_{2}}-\beta_{i_{2}})}}{\cos{(\alpha_{j_{2}}-\alpha_{i_{2}})}-\cos{(\beta_{j_{2}}-\beta_{i_{2}})}}\end{align*} and \begin{align*}-\frac{1+\cos{(\beta_{j_{1}}-\beta_{i_{1}})}}{\cos{(\alpha_{j_{1}}-\alpha_{i_{1}})}-\cos{(\beta_{j_{1}}-\beta_{i_{1}})}}=-\frac{1+\cos{(\beta_{j_{2}}-\beta_{i_{2}})}}{\cos{(\alpha_{j_{2}}-\alpha_{i_{2}})}-\cos{(\beta_{j_{2}}-\beta_{i_{2}})}},\end{align*}
    or
     \begin{align*}\frac{1-\cos{(\beta_{j_{1}}-\beta_{i_{1}})}}{\cos{(\alpha_{j_{1}}-\alpha_{i_{1}})}-\cos{(\beta_{j_{1}}-\beta_{i_{1}})}}=-\frac{1-\cos{(\beta_{j_{2}}-\beta_{i_{2}})}}{\cos{(\alpha_{j_{2}}-\alpha_{i_{2}})}-\cos{(\beta_{j_{2}}-\beta_{i_{2}})}}\end{align*} and \begin{align*}\frac{1+\cos{(\beta_{j_{1}}-\beta_{i_{1}})}}{\cos{(\alpha_{j_{1}}-\alpha_{i_{1}})}-\cos{(\beta_{j_{1}}-\beta_{i_{1}})}}=-\frac{1+\cos{(\beta_{j_{2}}-\beta_{i_{2}})}}{\cos{(\alpha_{j_{2}}-\alpha_{i_{2}})}-\cos{(\beta_{j_{2}}-\beta_{i_{2}})}}.\end{align*}
    The first of these possibilities is equivalent to \begin{align*}\cos{(\beta_{j_{1}}-\beta_{i_{1}})}=\cos{(\beta_{j_{2}}-\beta_{i_{2}})}\quad\textrm{and}\quad
      \cos{(\alpha_{j_{1}}-\alpha_{i})}=\cos{(\alpha_{j_{2}}-\alpha_{i})},
    \end{align*}
    while the second of these possibilities is equivalent to
    \begin{align*}\cos{(\beta_{j_{1}}-\beta_{i_{1}})}=-\cos{(\beta_{j_{2}}-\beta_{i_{2}})}\quad\textrm{and}\quad
      \cos{(\alpha_{j_{1}}-\alpha_{i})}=-\cos{(\alpha_{j_{2}}-\alpha_{i})}.
    \end{align*}
    This completes the proof.
  \end{proof}
  \begin{lemma}\label{Lemma 3}
    Let $q_{1},\textrm{ }\ldots,\textrm{ }q_{n}$ be a positive elliptic-elliptic rotopulsator for which $r_{i}$ and $\rho_{i}$ are independent of $i$ and not constant. Then for any $q_{i}$, $q_{j}$ for which $\beta_{i}=\beta_{j}\pm\pi$ we have that there are no $k\in\{1,\ldots,n\}$ for which $\beta_{k}=\beta_{i}$ and $\cos{(\alpha_{i}-\alpha_{k})}=-\cos{(\alpha_{j}-\alpha_{k})}$.
  \end{lemma}
  \begin{proof}
    If there are $k\in\{1,\ldots,n\}$ for which $\beta_{k}=\beta_{i}$ and \begin{align*}\cos{(\alpha_{i}-\alpha_{k})}=-\cos{(\alpha_{j}-\alpha_{k})}, \end{align*} then $\langle q_{i},q_{j}\rangle=-1$ if and only if $\sin{(\alpha_{i}-\alpha_{k})}=-\sin{(\alpha_{j}-\alpha_{k})}$, so because if $\langle q_{i},q_{j}\rangle=-1$ for a certain $i$, $j\in\{1,\ldots,n\}$, the right-hand side of (\ref{EquationsOfMotion Curved}) is undefined, we have that $\sin{(\alpha_{i}-\alpha_{k})}=\sin{(\alpha_{j}-\alpha_{k})}\neq 0$. Additionally, there might be a $\widehat{j}\in\{1,\ldots,n\}$ such that $\cos{(\alpha_{i}-\alpha_{k})}=\cos{(\alpha_{\widehat{j}}-\alpha_{k})}$ such that $\beta_{\widehat{j}}=\beta_{i}$ and $\alpha_{\widehat{j}}\neq\alpha_{i}$. If such a $\widehat{j}$ does not exist, then by (\ref{Crit1}) and Lemma~\ref{Lemma 2} we have that
    \begin{align*}
      0 &=\frac{m_{i}\sin{(\alpha_{i}-\alpha_{k})}}{(1-(1+r^{2}(\cos{(\alpha_{i}-\alpha_{k})}-1))^{2})^{\frac{3}{2}}}+\frac{m_{j}\sin{(\alpha_{j}-\alpha_{k})}}{(1-(-1+r^{2}(\cos{(\alpha_{j}-\alpha_{k})}+1))^{2})^{\frac{3}{2}}},
    \end{align*}
    which because $\sin{(\alpha_{i}-\alpha_{k})}=\sin{(\alpha_{j}-\alpha_{k})}\neq 0$ means that \begin{align*}0=m_{i}+m_{j}>0, \end{align*} which is a contradiction.

    If such a $\widehat{j}$ does exist, then because $\cos{(\alpha_{i}-\alpha_{k})}=\cos{(\alpha_{\widehat{j}}-\alpha_{k})}$ and $\widehat{j}\neq i$, we have that
    \begin{align*}
      \alpha_{\widehat{j}}-\alpha_{k}=2\pi-(\alpha_{i}-\alpha_{k}),
    \end{align*}
    which means that \begin{align*}\sin{(\alpha_{\widehat{j}}-\alpha_{k})}=-\sin{(\alpha_{i}-\alpha_{k})}. \end{align*} But as by construction we now have that \begin{align*}\cos{(\alpha_{\widehat{j}}-\alpha_{k})}=-\cos{(\alpha_{j}-\alpha_{k})}\textrm{,}\quad\cos{(\beta_{\widehat{j}}-\beta_{k})}=-\cos{(\beta_{j}-\beta_{k})}\end{align*} and \begin{align*}\sin{(\beta_{\widehat{j}}-\beta_{k})}=0=-\sin{(\beta_{j}-\beta_{k})}, \end{align*} to avoid the possibility of $\langle q_{j},q_{\widehat{j}}\rangle=-1$, we need that \begin{align*}\sin{(\alpha_{\widehat{j}}-\alpha_{k})}=\sin{(\alpha_{j}-\alpha_{k})}\neq 0. \end{align*} But because we also have that \begin{align*}\sin{(\alpha_{i}-\alpha_{k})}=\sin{(\alpha_{j}-\alpha_{k})}\quad\textrm{and}\quad\sin{(\alpha_{i}-\alpha_{k})}=-\sin{(\alpha_{\widehat{j}}-\alpha_{k})}, \end{align*} we get a contradiction.

    This means that our assumption that there are $k\in\{1,\ldots,n\}$ for which $\beta_{k}=\beta_{i}$ and $\cos{(\alpha_{i}-\alpha_{k})}=-\cos{(\alpha_{j}-\alpha_{k})}$ is false, which proves that there are no $k\in\{1,\ldots,n\}$ for which $\beta_{k}=\beta_{i}$ and $\cos{(\alpha_{i}-\alpha_{k})}=-\cos{(\alpha_{j}-\alpha_{k})}$.
  \end{proof}
  \section{Proof of Theorem~\ref{Main Theorem 1}}\label{Section proof of main theorem 1}
    We will assume that there are $i$, $j\in\{1,\ldots,n\}$ such that $\alpha_{j}-\alpha_{i}\neq 0$. If this is not the case, then we switch the roles of the $\alpha$s and the $\beta$s. Additionally, for $j$, $k\in\{1,...,n\}$, let
    \begin{align*}
      C_{\alpha_{j}\alpha_{k}\beta_{j}\beta_{k}}=\frac{\cos{(\alpha_{j}-\alpha_{k})}-\cos{(\beta_{j}-\beta_{k})}}{\left(1-\left(\cos{(\beta_{j}-\beta_{k})}+r^{2}\left(\cos{(\alpha_{j}-\alpha_{k})}-\cos{(\beta_{j}-\beta_{k})}\right)\right)^{2}\right)^{\frac{3}{2}}}
    \end{align*}
    and
    \begin{align*}
      S_{\alpha_{j}\alpha_{k}\beta_{j}\beta_{k}}=\frac{\sin{(\alpha_{j}-\alpha_{k})}}{\left(1-\left(\cos{(\beta_{j}-\beta_{k})}+r^{2}\left(\cos{(\alpha_{j}-\alpha_{k})}-\cos{(\beta_{j}-\beta_{k})}\right)\right)^{2}\right)^{\frac{3}{2}}}.
    \end{align*}
    Let $i\in\{1,\ldots,n\}$ and define $V_{i}=\{j\in\{1,\ldots,n\}\mid\beta_{j}=\beta_{i}\}$. For any $i\in\{1,\ldots,n\}$, $k$, $j\in V_{i}$, we now have that $\cos{(\beta_{j}-\beta_{k})}=1$ and by (\ref{Crit3}) that
    \begin{align}
      \delta &=r\rho^{2}\sum\limits_{j\notin V_{i}}m_{j}C_{\alpha_{j}\alpha_{k}\beta_{j}\beta_{k}}-r\rho^{2}\sum\limits_{j\in V_{i},\textrm{ }j\neq k}m_{j}C_{\alpha_{j}\alpha_{k}\beta_{i}\beta_{i}}.\label{COSV1}
    \end{align}
    By the same argument, we find by (\ref{Crit1}) that
    \begin{align}
      0 &=\sum\limits_{j\notin V_{i}}m_{j}S_{\alpha_{j}\alpha_{k}\beta_{j}\beta_{k}}+\sum\limits_{j\in V_{i},\textrm{ }j\neq k}m_{j}S_{\alpha_{j}\alpha_{k}\beta_{i}\beta_{i}}.\label{SINV1}
    \end{align}
    Now assume that the vectors $(q_{j1},q_{j2})^{T}$, $j\in V_{i}$, represent the vertices of an irregular polygon. Label the $\alpha_{j}$, $j\in V_{i}$, as $\alpha_{i_{1}}$, \ldots, $\alpha_{i_{u}}$ with $u=|V_{i}|$, $\alpha_{i_{u+1}}:=2\pi+\alpha_{i_{1}}$, $\alpha_{i_{s}}<\alpha_{i_{s+1}}$ and $\alpha_{i_{2}}-\alpha_{i_{1}}\leq\alpha_{i_{s+1}}-\alpha_{i_{s}}$ for all $s\in\{1,\ldots,u\}$. Subtracting (\ref{COSV1}) for $k=i_{1}$ from (\ref{COSV1}) for $k=i_{2}$, we get by Lemmas~\ref{Lemma 2} and \ref{Lemma 3} that
    \begin{align}
      0 &=-\sum\limits_{j\in V_{i},\textrm{ }j\neq i_{2}}m_{j}C_{\alpha_{j}\alpha_{i_{2}}\beta_{i}\beta_{i}}+\sum\limits_{j\in V_{i},\textrm{ }j\neq i_{1}}m_{j}C_{\alpha_{j}\alpha_{i_{1}}\beta_{i}\beta_{i}}\label{COSV2}
    \end{align}
    by Lemma~\ref{Lemma 2}, Lemma~\ref{Lemma 3} and (\ref{SINV1}), taking $k=i_{1}$ we get
    \begin{align}
      0 &=\sum\limits_{j\in V_{i},\textrm{ }j\neq i_{1}}m_{j}S_{\alpha_{j}\alpha_{i_{1}}\beta_{i}\beta_{i}}\label{SINV2 1}
    \end{align}
    and by Lemma~\ref{Lemma 2}, Lemma~\ref{Lemma 3} and (\ref{SINV1}), taking $k=i_{2}$ we get
    \begin{align}
      0 &=\sum\limits_{j\in V_{i},\textrm{ }j\neq i_{2}}m_{j}S_{\alpha_{j}\alpha_{i_{2}}\beta_{i}\beta_{i}}.\label{SINV2 2}
    \end{align}
    Note that there has to be a $v\in\{1,\ldots,u\}$ such that $\alpha_{i_{2}}-\alpha_{i_{1}}<\alpha_{i_{v+1}}-\alpha_{i_{v}}$, as otherwise $\alpha_{i_{s+1}}-\alpha_{i_{s}}$ has the same value for all $s\in\{1,\ldots,u\}$, in which case the vertices $(q_{j1},q_{j2})^{T}$, $j\in V_{i}$ represent the vertices of a regular polygon, which contradicts our assumption that they represent the vertices of an irregular polygon. Suppose for this $v$ that we have a $w$ such that \begin{align*}\cos{(\alpha_{i_{w}}-\alpha_{i_{2}})}=\cos{(\alpha_{i_{v}}-\alpha_{i_{1}})}.\end{align*} Then
    \begin{align*}
      \alpha_{i_{w}}-\alpha_{i_{2}}=\alpha_{i_{v}}-\alpha_{i_{1}},\textrm{ or }\alpha_{i_{w}}-\alpha_{i_{2}}=2\pi-(\alpha_{i_{v}}-\alpha_{i_{1}}).
    \end{align*}
    If the first of these two identities is true, then
    \begin{align*}
      \alpha_{i_{w}}-\alpha_{i_{v}}=\alpha_{i_{2}}-\alpha_{i_{1}}>0,
    \end{align*}
    so $w>v$ and therefore $w\geq v+1$, which by construction means that
    \begin{align*}
      \alpha_{i_{2}}-\alpha_{i_{1}}<\alpha_{i_{v+1}}-\alpha_{i_{v}}\leq \alpha_{i_{w}}-\alpha_{i_{v}}=\alpha_{i_{2}}-\alpha_{i_{1}},
    \end{align*}
    which is a contradiction.
    Our proof now comes down to checking all possible ways a term $m_{j}C_{\alpha_{v}\alpha_{i_{1}}\beta_{i}\beta_{i}}$
    can cancel out against other terms in (\ref{COSV2}) under the conditions imposed by (\ref{SINV2 1}) and (\ref{SINV2 2}). \\
    So by Lemmas~\ref{Lemma 2} and \ref{Lemma 3} the only terms against which $C_{i_{1}i_{v}}$ can cancel out in (\ref{COSV2}) are terms $C_{\alpha_{i_{1}}\alpha_{i_{a}}\beta_{i}\beta_{i}}$ for which \begin{align*}\cos{(\alpha_{i_{v}}-\alpha_{i_{1}})}=\cos{(\alpha_{i_{a}}-\alpha_{i_{1}})}\end{align*} and $C_{i_{2}i_{p}}$ for which \begin{align*}\cos{(\alpha_{i_{v}}-\alpha_{i_{1}})}=\cos{(\alpha_{i_{p}}-\alpha_{i_{2}})}, \end{align*} where $a$, $p$, $w\in\{1,\ldots,u\}$. \\
    Note that $\cos{(\alpha_{i_{v}}-\alpha_{i_{1}})}=\cos{(\alpha_{i_{a}}-\alpha_{i_{1}})}$ if and only if $\alpha_{i_{a}}-\alpha_{i_{1}}=\alpha_{i_{v}}-\alpha_{i_{1}}$, or $\alpha_{i_{a}}-\alpha_{i_{1}}=2\pi-(\alpha_{i_{v}}-\alpha_{i_{1}})$, which means that $\alpha_{i_{a}}-\alpha_{i_{1}}=2\pi-(\alpha_{i_{v}}-\alpha_{i_{1}})$, as otherwise $i_{a}=i_{v}$.
    Additionally, note by the same argument that if $C_{\alpha_{i_{1}}\alpha_{i_{a}}\beta_{i}\beta_{i}}$ exists, that $\sin{(\alpha_{i_{v}}-\alpha_{i_{1}})}\neq 0$, as otherwise $\langle q_{i_{v}}, q_{i_{a}}\rangle^{2}=1$, which makes the right-hand side of (\ref{EquationsOfMotion Curved}) undefined. This means by (\ref{SINV2 1}), Lemma~\ref{Lemma 2} and Lemma~\ref{Lemma 3} that if $C_{\alpha_{i_{1}}\alpha_{i_{a}}\beta_{i}\beta_{i}}$ exists, we have that $0=m_{i_{v}}-m_{i_{a}}$ and if $C_{\alpha_{i_{1}}\alpha_{i_{a}}\beta_{i}\beta_{i}}$ does not exist, then $0=m_{i_{v}}$, which is impossible, as $m_{i_{v}}>0$. Finally, note that if $C_{\alpha_{i_{2}}\alpha_{i_{p}}\beta_{i}\beta_{i}}$ exists, we have that $\sin{(\alpha_{i_{p}}-\alpha_{i_{2}})}\neq 0$, as otherwise $\alpha_{i_{p}}=\alpha_{i_{2}}$, or $\alpha_{i_{p}}-\alpha_{i_{2}}=\pi$, which, as $\alpha_{i_{p}}\neq\alpha_{i_{2}}$, means that $\pi=\alpha_{i_{p}}-\alpha_{i_{2}}=2\pi-(\alpha_{i_{v}}-\alpha_{i_{1}})$, which means that $\alpha_{i_{p}}-\alpha_{i_{2}}=\alpha_{i_{v}}-\alpha_{i_{1}}$, which means that $\alpha_{i_{p}}-\alpha_{i_{v}}=\alpha_{i_{2}}-\alpha_{i_{1}}>0$, which means that $p>v$, which means that $p\geq v+1$, which means that \begin{align*}\alpha_{i_{2}}-\alpha_{i_{1}}=\alpha_{i_{p}}-\alpha_{i_{v}}\geq\alpha_{i_{v+1}}-\alpha_{i_{v}}>\alpha_{i_{2}}-\alpha_{i_{1}}, \end{align*}
    which is a contradiction. But that then means by (\ref{SINV2 2}) that $0=m_{i_{p}}$, which is impossible, as $m_{i_{p}}>0$. This then finally means by (\ref{COSV2}) that because $C_{\alpha_{i_{1}}\alpha_{i_{a}}\beta_{i}\beta_{i}}$ exists, we have by Lemmas~\ref{Lemma 2} and \ref{Lemma 3} that $0=2m_{i_{v}}$ which is a contradiction, because all masses are positive. We can therefore finally conclude that our assumption that the $(q_{j1},q_{j2})^{T}$, $j\in V_{i}$ can represent the vertices of an irregular polygon is false. This completes the proof that the $(q_{j1},q_{j2})^{T}$, $j\in V_{i}$, represent vertices of a regular polygon. Next we will prove that for any $l$, $k\in\{1,\ldots,n\}$, $V_{l}\neq V_{k}$, we have that $|V_{l}|=|V_{k}|$: If $|V_{l}|\neq |V_{k}|$, $l_{1}\in V_{i}$, $k_{1}\in V_{k}$, then by Lemmas~\ref{Lemma 2} and \ref{Lemma 3} and subtracting (\ref{Crit3}) for $i=l_{1}$ from (\ref{Crit3}) for $i=k_{1}$, we have that
    \begin{align}
      0 &=-\sum\limits_{j\in V_{k}}m_{j}C_{\alpha_{k_{1}}\alpha_{j}\beta_{k}\beta_{k}}+\sum\limits_{j\in V_{l}}m_{j}C_{\alpha_{l_{1}}\alpha_{j}\beta_{l}\beta_{l}}.\label{COSV4}
    \end{align}
    Let $\frac{2\pi}{|V_{k}|}=\alpha_{j_{k}}-\alpha_{k_{1}}$ be the smallest possible positive angle for $j_{k}$, $k_{1}\in V_{k}$ and let $\frac{2\pi}{|V_{l}|}=\alpha_{j_{l}}-\alpha_{l_{1}}$ be the smallest possible positive angle for $j_{l}$, $l_{1}\in V_{l}$. Then by Lemmas~\ref{Lemma 2} and \ref{Lemma 3}, the only way a $C_{\alpha_{k_{1}}\alpha_{j}\beta_{k}\beta_{k}}$ in (\ref{COSV4}) can cancel out, is if there is a $C_{\alpha_{l_{1}}\alpha_{\widehat{j}_{l}}\beta_{l}\beta_{l}}$, $\widehat{j}_{l}\in V_{l}$ in (\ref{COSV4}) to cancel out against. But if $|V_{k}|\neq |V_{l}|$, then we may assume that $|V_{k}|>|V_{l}|$, in which case there is no $j\in V_{l}$ for which $\cos{(\alpha_{j_{k}}-\alpha_{k_{1}})}=\cos{(\alpha_{j}-\alpha_{l_{1}})}$, which contradicts (\ref{COSV4}). This thus proves that the $(q_{j1},q_{j2})^{T}$ represent the vertices of a regular polygon as long as the corresponding $\beta_{j}$ are equal, that polygons related to different $\beta_{j}$ have equal numbers of vertices, that the same is true for the $(q_{j3},q_{j4})^{T}$ if the $\alpha_{j}$ are equal and that polygons related to different $\alpha_{j}$ have equal numbers of vertices as well. However, it may still be possible that two polygons related to two different $\beta_{j}$ do not have the same vertices and that in that case the the combined set of those vertices are not vertices of a regular polygon, or that the combined set of vertices of two polygons related to two different $\alpha_{j}$ do not represent vertices of a regular polygon. We will prove that that possibility can be excluded next:

    If we write $|V_{i}|=u_{\beta}$ and the number of distinct $\beta_{i}$ is $\widehat{p}_{\beta}$, then there might be $\gamma_{1}$, \ldots, $\gamma_{\widehat{p}_{\beta}}$, $0\leq\gamma_{1}\leq\ldots\leq\gamma_{\widehat{p}_{\beta}}$, such that for all $j\in V_{i}$ there is an $l\in\{1,\ldots,u_{\beta}\}$ such that $\alpha_{j}=\frac{2\pi l}{u_{\beta}}+\gamma_{i}$. What we will prove is that the $\gamma_{i}$ have to all be equal, which then means that the $(q_{j1},q_{j2})^{T}$ do not only represent the vertices of a regular polygon for $j\in V_{i}$, but that the $(q_{j1},q_{j2})^{T}$ actually represent the vertices of a regular polygon for $j\in\{1,\ldots,n\}$ and by reusing that argument for the $(q_{j3},q_{j4})^{T}$ instead, we have that the $(q_{j3},q_{j4})^{T}$ actually represent the vertices of a regular polygon for $j\in\{1,\ldots,n\}$: By (\ref{Crit2}) and Lemma~\ref{Lemma 2} we have for all $i$, $l_{1}\in\{1,\ldots,\widehat{p}_{\beta}\}$, $i\neq l_{1}$ that $\sin{(\beta_{l_{1}}-\beta_{i})}=0$, or that there is an $l_{2}\in\{1,\ldots,\widehat{p}_{\beta}\}$ such that $\sin{(\beta_{l_{1}}-\beta_{i})}=-\sin{(\beta_{l_{2}}-\beta_{i})}$. If $\sin{(\beta_{l_{1}}-\beta_{i})}=0$, then $\beta_{i}=\beta_{l_{1}}$, which can be ignored by construction, or $\beta_{i}=\beta_{l_{1}}\pm\pi$, which can be ignored by Lemma~\ref{Lemma 3}. So that means by Lemma~\ref{Lemma 2} that, as  there is an $l_{2}\in\{1,\ldots,\widehat{p}_{\beta}\}$ such that $\sin{(\beta_{l_{1}}-\beta_{i})}=-\sin{(\beta_{l_{2}}-\beta_{i})}$, that either
    \begin{align*}
      \cos{\left(\frac{2\pi l}{u_{\beta}}+\gamma_{l_{1}}-\gamma_{i}\right)}=\cos{\left(\frac{2\pi l}{u_{\beta}}+\gamma_{l_{2}}-\gamma_{i}\right)} \quad
      \textrm{and}\quad\cos{\left(\beta_{i}-\beta_{l_{1}}\right)}=\cos{\left(\beta_{i}-\beta_{l_{2}}\right)},
    \end{align*}
    or
    \begin{align*}
      \cos{\left(\frac{2\pi l}{u_{\beta}}+\gamma_{l_{1}}-\gamma_{i}\right)}=-\cos{\left(\frac{2\pi l}{u_{\beta}}+\gamma_{l_{2}}-\gamma_{i}\right)}\quad
      \textrm{and}\quad\cos{\left(\beta_{i}-\beta_{l_{1}}\right)}=-\cos{\left(\beta_{i}-\beta_{l_{2}}\right)}
    \end{align*}
    for all $l\in\{1,\ldots,u_{\beta}\}$. Writing out the first of these four equations gives that $\gamma_{l_{1}}=\gamma_{l_{1}}$, or
    \begin{align*}
      \frac{2\pi l}{u_{\beta}}+\gamma_{l_{1}}-\gamma_{i}=-\left(\frac{2\pi l}{u_{\beta}}+\gamma_{l_{2}}-\gamma_{i}\right)\textrm{ }(\textrm{mod }2\pi).
    \end{align*}
    The first of these possibilities means that indeed the $\gamma_{i}$ are equal and the second means that
    \begin{align*}
      \gamma_{l_{1}}+\gamma_{l_{2}}-2\gamma_{i}=-\frac{4\pi l}{u_{\beta}}\textrm{ }(\textrm{mod }2\pi)
    \end{align*}
    for all $l\in\{1,\ldots,u_{\beta}\}$, which means that $\gamma_{l_{1}}+\gamma_{l_{2}}-2\gamma_{i}$ is multivalued, which is a contradiction, or $u_{\beta}=2$, in which case by Lemma~\ref{Lemma 2}, (\ref{Crit1}) and (\ref{Crit2}) there are $\widehat{i}$, $\widehat{j}$, $k\in\{1,\ldots,n\}$ such that \begin{align*}
      \sin{(\alpha_{\widehat{j}}-\alpha_{\widehat{i}})}=\sin{(\alpha_{k}-\alpha_{\widehat{i}})}\sin{(\alpha_{k}-\alpha_{\widehat{j}})}=0
    \end{align*}
    and
    \begin{align*}
      \sin{(\beta_{\widehat{j}}-\beta_{\widehat{i}})}=\sin{(\beta_{k}-\beta_{\widehat{i}})}\sin{(\beta_{k}-\alpha_{\widehat{j}})}=0,
    \end{align*}
    which means that $\langle q_{\widehat{i}},q_{\widehat{j}}\rangle^{2}=1$, $\langle q_{\widehat{i}},q_{k}\rangle^{2}=1$, or $\langle q_{k},q_{\widehat{j}}\rangle^{2}=1$, which means that the right-hand side of (\ref{EquationsOfMotion Curved}) is undefined for $i=\widehat{i}$, $i=\widehat{j}$, or $i=k$, which is a contradiction.

    Writing out
    \begin{align}\label{HELP}
      \cos{\left(\frac{2\pi l}{u_{\beta}}+\gamma_{l_{1}}-\gamma_{i}\right)}=-\cos{\left(\frac{2\pi l}{u_{\beta}}+\gamma_{l_{2}}-\gamma_{i}\right)}
    \end{align}
    gives that $\gamma_{l_{1}}=\gamma_{l_{2}}\pm\pi$, or
    \begin{align*}
      \frac{2\pi l}{u_{\beta}}+\gamma_{l_{1}}-\gamma_{i}=\pi-\left(\frac{2\pi l}{u_{\beta}}+\gamma_{l_{2}}-\gamma_{i}\right)\textrm{ }(\textrm{mod }2\pi).
    \end{align*}
    The first of these identities means by (\ref{HELP}), Lemma~\ref{Lemma 2} and the fact that \begin{align*}\sin{(\beta_{l_{1}}-\beta_{i})}=-\sin{(\beta_{l_{2}}-\beta_{i})}\end{align*} that there are $j_{1}$, $j_{2}\in\{1,\ldots,n\}$ such that $\langle q_{j_{1}},q_{j_{2}}\rangle=-1$, which means that for $i=j_{1}$ the right-hand side of (\ref{EquationsOfMotion Curved}) is undefined.
    The second of these identities means that
    \begin{align*}
      \gamma_{l_{1}}+\gamma_{l_{2}}-2\gamma_{i}=\pi-\frac{4\pi l}{u_{\beta}}\textrm{ }(\textrm{mod }2\pi)
    \end{align*}
    for all $l\in\{1,\ldots,u_{\beta}\}$, which again means that $\gamma_{l_{1}}+\gamma_{l_{2}}-2\gamma_{i}$ is multivalued, which is a contradiction, or that $u_{\beta}=2$, which again means that there are $\widehat{i}$, $\widehat{j}$, $k\in\{1,\ldots,n\}$ such that $\langle q_{\widehat{i}},q_{\widehat{j}}\rangle^{2}=1$, $\langle q_{\widehat{i}},q_{k}\rangle^{2}=1$, or $\langle q_{k},q_{\widehat{j}}\rangle^{2}=1$, which is again a contradiction. This finally proves that the $(q_{j1},q_{j2})^{T}$ actually represent the vertices of a regular polygon of $u_{\beta}$ vertices for $j\in\{1,\ldots,n\}$ and by reusing that argument for the $(q_{j3},q_{j4})^{T}$ instead, we have that the $(q_{j3},q_{j4})^{T}$ actually represent the vertices of a regular polygon for $j\in\{1,\ldots,n\}$.
    What remains to be proven is that all the masses are equal:
    Let $i_{1}$, $i_{2}\in V_{i}$, $i_{1}\neq i_{2}$. Then because the $(q_{j1},q_{j2})^{T}$, $j\in V_{i}$ represent vertices of a regular polygon, there are $j_{1}$, $j_{2}\in V_{i}$ such that
    \begin{align*}
      \cos{(\alpha_{j_{1}}-\alpha_{i_{1}})}=\cos{(\alpha_{j_{2}}-\alpha_{i_{2}})}=\cos{(\alpha_{i_{1}}-\alpha_{i_{2}})},
    \end{align*} which means that subtracting (\ref{Crit3}) for $i=i_{1}$ from (\ref{Crit3}) for $i=i_{2}$ gives by Lemmas~\ref{Lemma 2} and \ref{Lemma 3} that
    \begin{align}\label{Mass1}
      0=(m_{i_{1}}-m_{i_{2}}+m_{j_{1}}-m_{j_{2}})\frac{1-\cos{(\alpha_{i_{1}}-\alpha_{i_{2}})}}{(1-(1+r^{2}(\cos{(\alpha_{i_{1}}-\alpha_{i_{2}})}-1))^{2})^{\frac{3}{2}}},
    \end{align}
    by Lemma~\ref{Lemma 2}, Lemma~\ref{Lemma 3} and (\ref{Crit1}) for $i=i_{1}$ that
    \begin{align}\label{Mass2}
      0=(m_{i_{1}}-m_{j_{1}})\frac{\sin{(\alpha_{i_{1}}-\alpha_{i_{2}})}}{(1-(1+r^{2}(\cos{(\alpha_{i_{1}}-\alpha_{i_{2}})}-1))^{2})^{\frac{3}{2}}},
    \end{align}
    and by Lemma~\ref{Lemma 2}, Lemma~\ref{Lemma 3} and (\ref{Crit1}) for $i=i_{2}$ that
    \begin{align}\label{Mass3}
      0=(m_{i_{2}}-m_{j_{2}})\frac{\sin{(\alpha_{i_{1}}-\alpha_{i_{2}})}}{(1-(1+r^{2}(\cos{(\alpha_{i_{1}}-\alpha_{i_{2}})}-1))^{2})^{\frac{3}{2}}},
    \end{align}
    so combining (\ref{Mass1})--(\ref{Mass3}), we find that $m_{i_{1}}=m_{i_{2}}$ if $\sin{(\alpha_{i_{1}}-\alpha_{i_{2}})}\neq 0$. As for any $i_{1}\in V_{i}$ there is by construction at most one $i_{2}\in V_{i}$, $i_{2}\neq i_{1}$, for which $\sin{(\alpha_{i_{1}}-\alpha_{i_{2}})}=0$, this means that $m_{i_{1}}=m_{i_{2}}$ for all $i_{1}$, $i_{2}\in V_{i}$.

    Now let $m_{j}=M_{i}$ for all $j\in V_{i}$, let $i_{1}$, $i_{2}\in\{1,\ldots,n\}$, $V_{i_{1}}\neq V_{i_{2}}$, $j_{1}\in V_{i_{1}}$ and $j_{2}\in V_{i_{2}}$. Then subtracting (\ref{Crit3}) for $i=j_{1}$ from (\ref{Crit3}) for $i=j_{2}$ gives, again by Lemmas~\ref{Lemma 2} and \ref{Lemma 3},
    \begin{align*}
      0&=M_{i_{2}}\sum\limits_{j\in V_{i_{2}},\textrm{ }j\neq j_{2}}\frac{1-\cos{(\alpha_{j}-\alpha_{j_{2}})}}{(1-(1+r^{2}(\cos{(\alpha_{j}-\alpha_{j_{2}})}-1))^{2})^{\frac{3}{2}}}\\
      &-M_{i_{1}}\sum\limits_{j\in V_{i_{1}},\textrm{ }j\neq j_{1}}\frac{1-\cos{(\alpha_{j}-\alpha_{j_{1}})}}{(1-(1+r^{2}(\cos{(\alpha_{j}-\alpha_{j_{1}})}-1))^{2})^{\frac{3}{2}}},
    \end{align*}
    which as the $(q_{j1},q_{j2})^{T}$ represent the vertices of a regular polygon of $u_{\beta}$ vertices can be rewritten as
    \begin{align*}
      0&=M_{i_{2}}\sum\limits_{\substack{j=1\\j\neq j_{2}}}^{u_{\beta}}\frac{1-\cos{\frac{2\pi(j-j_{2})}{u_{\beta}}}}{\left(1-\left(1+r^{2}\left(\cos{\frac{2\pi(j-j_{2})}{u_{\beta}}}-1\right)\right)^{2}\right)^{\frac{3}{2}}}\\
      &-M_{i_{1}}\sum\limits_{\substack{j=1\\j\neq j_{1}}}^{u_{\beta}}\frac{1-\cos{\frac{2\pi(j-j_{1})}{u_{\beta}}}}{\left(1-\left(1+r^{2}\left(\cos{\frac{2\pi(j-j_{1})}{u_{\beta}}}-1\right)\right)^{2}\right)^{\frac{3}{2}}}.
    \end{align*}
    Rewriting the first of these two sums in terms of $k_{2}=j-j_{2}$ instead of $j$ and the second sum in terms of $k_{1}=j-j_{1}$ instead of $j$ and then replacing $k_{2}$ and $k_{1}$ with $j$ then gives
    \begin{align*}
      0&=M_{i_{2}}\sum\limits_{j=1}^{u_{\beta}-1}\frac{1-\cos{\frac{2\pi j}{u_{\beta}}}}{\left(1-\left(1+r^{2}\left(\cos{\frac{2\pi j}{u_{\beta}}}-1\right)\right)^{2}\right)^{\frac{3}{2}}}\\
      &-M_{i_{1}}\sum\limits_{j=1}^{u_{\beta}-1}\frac{1-\cos{\frac{2\pi j}{u_{\beta}}}}{\left(1-\left(1+r^{2}\left(\cos{\frac{2\pi j}{u_{\beta}}}-1\right)\right)^{2}\right)^{\frac{3}{2}}}\\
      &=(M_{i_{2}}-M_{i_{1}})\sum\limits_{j=1}^{u_{\beta}-1}\frac{1-\cos{\frac{2\pi j}{u_{\beta}}}}{\left(1-\left(1+r^{2}\left(\cos{\frac{2\pi j}{u_{\beta}}}-1\right)\right)^{2}\right)^{\frac{3}{2}}},
    \end{align*}
    so $M_{i_{2}}=M_{i_{1}}$. This then finally proves that all masses are equal.
  \section{Proof of Theorem~\ref{Main Theorem 2}}\label{Section proof of main theorem 2}
      If there exists a positive elliptic-elliptic rotopulsator solution of (\ref{EquationsOfMotion Curved}), $r_{i}=r$ and $\rho_{i}=\rho$ independent of $i$ and not constant, for which $\|q_{i}(t)-q_{j}(t)\|$ is constant for all $i$, $j\in\{1,\ldots,n\}$, then we have that $\|q_{i}-q_{j}\|^{2}=\widehat{C}_{ij}$ for all $i$, $j\in\{1,\ldots,n\}$, for certain constants $\widehat{C}_{ij}$, which means that for all $i$, $j\in\{1,\ldots,n\}$ we have, using that $r^{2}+\rho^{2}=1$ and writing out $\|q_{i}-q_{j}\|^{2}$ using the expressions for $q_{i}$ and $q_{j}$, that
      \begin{align}
        \widehat{C}_{ij}&=r^{2}\left((\cos{(\alpha_{j}-\alpha_{i})}-1)^{2}+\sin^{2}{(\alpha_{j}-\alpha_{i})}\right)\nonumber\\
        &+\rho^{2}\left((\cos{(\beta_{j}-\beta_{i})}-1)^{2}+\sin^{2}{(\beta_{j}-\beta_{i})}\right)\nonumber\\
        &=2r^{2}\left(\cos{(\beta_{j}-\beta_{i})}-\cos{(\alpha_{j}-\alpha_{i})}\right)+2\left(1-\cos{(\beta_{j}-\beta_{i})}\right). \label{ConfigurationFixedEquation}
      \end{align}
      As $r$ is not constant, for (\ref{ConfigurationFixedEquation}) to hold, we need that \begin{align*}\cos{(\alpha_{j}-\alpha_{i})}-\cos{(\beta_{j}-\beta_{i})}= 0\end{align*} for all $i$, $j\in\{1,\ldots,n\}$, $i\neq j$.

      Suppose that there are $i$, $j$, $k\in\{1,\ldots,n\}$, $i\neq j$, $i\neq k$, $j\neq k$, such that $\alpha_{j}-\alpha_{i}=\beta_{j}-\beta_{i}\textrm{ }(\textrm{mod } 2\pi)$ and $\alpha_{j}-\alpha_{k}=-(\beta_{j}-\beta_{k})\textrm{ }(\textrm{mod } 2\pi)$. Then
      \begin{align}\label{alphaSUM1}
        \alpha_{k}-\alpha_{i}=(\beta_{j}-\beta_{k})+(\beta_{j}-\beta_{i})\textrm{ }(\textrm{mod } 2\pi)
      \end{align}
      and
      \begin{align}\label{betaSUM1}
        \beta_{k}-\beta_{i}=(\alpha_{j}-\alpha_{k})+(\alpha_{j}-\alpha_{i})\textrm{ }(\textrm{mod } 2\pi).
      \end{align}
      Using that then $\cos{(\alpha_{j}-\alpha_{i})}=\cos{(\beta_{j}-\beta_{i})}$, $\cos{(\alpha_{j}-\alpha_{k})}=\cos{(\beta_{j}-\beta_{k})}$, $\sin{(\alpha_{j}-\alpha_{i})}=\sin{(\beta_{j}-\beta_{i})}$ and $\sin{(\alpha_{j}-\alpha_{k})}=-\sin{(\beta_{j}-\beta_{k})}$, we find by taking cosines on both sides of (\ref{alphaSUM1}) and (\ref{betaSUM1}) that
      \begin{align}\label{alphaSUM2}
        \cos{(\alpha_{k}-\alpha_{i})}=\cos{(\beta_{j}-\beta_{k})}\cos{(\beta_{j}-\beta_{i})}-\sin{(\beta_{j}-\beta_{k})}\sin{(\beta_{j}-\beta_{i})}
      \end{align}
      and
      \begin{align}\label{betaSUM2}
        \cos{(\beta_{k}-\beta_{i})}&=\cos{(\alpha_{j}-\alpha_{k})}\cos{(\alpha_{j}-\alpha_{i})}-\sin{(\alpha_{j}-\alpha_{k})}\sin{(\alpha_{j}-\alpha_{i})}\nonumber\\
        &=\cos{(\beta_{j}-\beta_{k})}\cos{(\beta_{j}-\beta_{i})}+\sin{(\beta_{j}-\beta_{k})}\sin{(\beta_{j}-\beta_{i})}
      \end{align}
      respectively.
      Because $\cos{(\alpha_{k}-\alpha_{i})}=\cos{(\beta_{k}-\beta_{i})}$, comparing the righthand sides of (\ref{alphaSUM2}) and (\ref{betaSUM2}) now gives that
      \begin{align*}
        \sin{(\beta_{j}-\beta_{k})}\sin{(\beta_{j}-\beta_{i})}=0.
      \end{align*}
      We cannot have that $\beta_{j}=\beta_{k}$, or $\beta_{j}=\beta_{i}$, as $\alpha_{j}-\alpha_{i}=\beta_{j}-\beta_{i}\textrm{ }(\textrm{mod } 2\pi)$ and $\alpha_{j}-\alpha_{k}=-(\beta_{j}-\beta_{k})\textrm{ }(\textrm{mod } 2\pi)$, which would by construction mean that $q_{j}=q_{k}$, or $q_{j}=q_{i}$. So $\beta_{j}-\beta_{k}=\pi\textrm{ }(\textrm{mod } 2\pi)$, giving $\alpha_{j}-\alpha_{k}=\pi\textrm{ }(\textrm{mod } 2\pi)$ and therefore $q_{j}=-q_{k}$, or $\beta_{j}-\beta_{i}=\pi\textrm{ }(\textrm{mod } 2\pi)$, leading to $\alpha_{j}-\alpha_{i}=\pi\textrm{ }(\textrm{mod } 2\pi)$ and therefore $q_{j}=-q_{i}$. But then there are $i$, $j\in\{1,\ldots,n\}$, $i\neq j$, such that $\langle q_{i},q_{j}\rangle^{2}=1$, in which case there is a term in the sum on the righthand side of (\ref{EquationsOfMotion Curved}) that is undefined, which is a contradiction. We therefore conclude that for all $i$, $j\in\{1,\ldots,n\}$ we have that \begin{align*}\cos{(\alpha_{j}-\alpha_{i})}=\cos{(\beta_{j}-\beta_{i})}\quad\textrm{and}\quad\sin{(\alpha_{j}-\alpha_{i})}=\sin{(\beta_{j}-\beta_{i})}.\end{align*}
      We may thus assume that $\alpha_{i}=\beta_{i}$ for all $\{1,\ldots,n\}$.
      It therefore follows that for all $j\in\{1,\ldots,n\}$ we have that
      \begin{align*}
        q_{j}=\begin{pmatrix}
          rR(\theta)\begin{pmatrix}\cos{\alpha_{j}}\\
          \sin{\alpha_{j}}\end{pmatrix}\\
          \rho R(\phi)\begin{pmatrix}\cos{\alpha_{j}}\\
          \sin{\alpha_{j}}\end{pmatrix}
        \end{pmatrix}=(\cos{\alpha_{j}})\begin{pmatrix}
          rR(\theta)\begin{pmatrix}1\\
          0\end{pmatrix}\\
          \rho R(\phi)\begin{pmatrix}1\\
          0\end{pmatrix}
        \end{pmatrix}+(\sin{\alpha_{j}})\begin{pmatrix}
          rR(\theta)\begin{pmatrix}0\\
          1\end{pmatrix}\\
          \rho R(\phi)\begin{pmatrix}0\\
          1\end{pmatrix}
        \end{pmatrix}.
      \end{align*}
      Hence all the $q_{j}$ are spanned by the linearly independent vectors
      \begin{align*}
        \begin{pmatrix}
          rR(\theta)\begin{pmatrix}1\\
          0\end{pmatrix}\\
          \rho R(\phi)\begin{pmatrix}1\\
          0\end{pmatrix}
        \end{pmatrix}\textrm{ and }\begin{pmatrix}
          rR(\theta)\begin{pmatrix}0\\
          1\end{pmatrix}\\
          \rho R(\phi)\begin{pmatrix}0\\
          1\end{pmatrix}
        \end{pmatrix}.
      \end{align*}
      This then finally shows that all the $q_{j}$ have to lie in the same (rotating) plane, which is only possible if all the point masses lie on a great circle and the point masses represent vertices of a polygon.

      What remains is to show that positive elliptic-elliptic rotopulsators for which $r$ and $\rho$ are not constant and $\|q_{i}(t)-q_{j}(t)\|$ are constant for all $i$, $j\in\{1,\ldots,n\}$ exist: This comes down to finding out under which conditions on the $\alpha_{i}$, $i\in\{1,\ldots,n\}$, $\beta_{i}$, $i\in\{1,\ldots,n\}$, we have that (\ref{Crit1})--(\ref{Crit3}) are met. As for any such rotopulsator we have that $\alpha_{i}=\beta_{i}$ for all $i\in\{1,\ldots,n\}$, proving that these three identities are met reduces by Lemma~\ref{Lemma 1} and because $\cos{(\alpha_{j}-\alpha_{i})}=\cos{(\beta_{j}-\beta_{i})}$ for all $i$, $j\in\{1,\ldots,n\}$ to proving that there exist constants $C_{1}$, $C_{2}$, $\alpha_{1},\textrm{ }\ldots,\textrm{ }\alpha_{n}$ such that
      \begin{align}
      0=\sum\limits_{\substack{j=1\\j\neq i}}^{n}\frac{m_{j}\sin{(\alpha_{j}-\alpha_{i})}}{\left|\sin{(\alpha_{j}-\alpha_{i})}\right|^{3}}\quad
      \textrm{and}\quad\ddot{r}+r\rho^{2}\left(\frac{C_{2}^{2}}{\rho^{4}}-\frac{C_{1}^{2}}{r^{4}}\right)+r\left(\frac{\dot{r}}{\rho}\right)^{2}=0. \label{Crit3C}
    \end{align}
    Note that
    \begin{align*}
      \frac{d}{dt}\left((\dot{r})^{2}+(\dot{\rho})^{2}+\frac{C_{1}^{2}}{r^{2}}+\frac{C_{2}^{2}}{\rho^{2}}\right)=\frac{2\dot{r}}{\rho^{2}}\left(\ddot{r}+r\rho^{2}\left(\frac{C_{2}^{2}}{\rho^{4}}-\frac{C_{1}^{2}}{r^{4}}\right)+r\left(\frac{\dot{r}}{\rho}\right)^{2}\right),
    \end{align*}
    so by (\ref{Crit3C}) we have that
    \begin{align*}
      (\dot{r})^{2}+(\dot{\rho})^{2}+\frac{C_{1}^{2}}{r^{2}}+\frac{C_{2}^{2}}{\rho^{2}}=C
    \end{align*}
    for a suitable constant $C$, which, using that $r\dot{r}+\rho\dot{\rho}=0$, can be rewritten as
    \begin{align*}
      \left(\frac{\dot{r}}{\rho}\right)^{2}+\frac{C_{1}^{2}}{r^{2}}+\frac{C_{2}^{2}}{\rho^{2}}=C,
    \end{align*}
    which is a first-order ordinary differential equation, which has nonconstant solutions $r$ for suitable choices of $C$, $C_{1}$ and $C_{2}$. So that means that (\ref{Crit3C}) holds true for nonconstant $r$. Finally, the set of equations described by (\ref{Crit3C}) is the exact set of equations used to prove Theorem~4 in \cite{D2}, where it was proven that even for $n=3$ the configuration of the point masses need not be a regular polygon. Proving that for general $n$ solutions exist can be done by choosing $\alpha_{j}=\frac{2\pi j}{n}$, $j\in\{1,\ldots,n\}$ and all masses equal and repeat the proof in \cite{D2} of Theorem~1, or observing that the right-hand side of the first formula in (\ref{Crit3C}) then becomes
    \begin{align}
      \sum\limits_{\substack{j=1\\j\neq i}}^{n}\frac{m_{j}\sin{(\alpha_{j}-\alpha_{i})}}{\left|\sin{(\alpha_{j}-\alpha_{i})}\right|^{3}}=\sum\limits_{\substack{j=1\\j\neq i}}^{n}\frac{m_{j}\sin{\frac{2\pi(j-i)}{n}}}{\left|\sin{\frac{2\pi(j-i)}{n}}\right|^{3}}. \label{SineGreatCircle}
    \end{align}
    Rewriting (\ref{SineGreatCircle}) in terms of $\widehat{j}=j-i$ and then replacing $\widehat{j}$ with $j$ and writing $m_{j}=m$ gives
    \begin{align}
      \sum\limits_{\substack{j=1\\j\neq i}}^{n}\frac{m_{j}\sin{(\alpha_{j}-\alpha_{i})}}{\left|\sin{(\alpha_{j}-\alpha_{i})}\right|^{3}}=\sum\limits_{j=1}^{n-1}\frac{m\sin{\frac{2\pi j}{n}}}{\left|\sin{\frac{2\pi j}{n}}\right|^{3}}\label{SineGreatCircle2}
    \end{align}
    and as
    \begin{align*}
      \sum\limits_{j=1}^{n-1}\frac{m\sin{\frac{2\pi j}{n}}}{\left|\sin{\frac{2\pi j}{n}}\right|^{3}}=\sum\limits_{j=1}^{n-1}\frac{m\sin{\frac{2\pi (n-j)}{n}}}{\left|\sin{\frac{2\pi(n-j)}{n}}\right|^{3}}=-\sum\limits_{j=1}^{n-1}\frac{m\sin{\frac{2\pi j}{n}}}{\left|\sin{\frac{2\pi j}{n}}\right|^{3}},
    \end{align*}
    that means by (\ref{SineGreatCircle2}) that
    \begin{align*}
      \sum\limits_{\substack{j=1\\j\neq i}}^{n}\frac{m_{j}\sin{(\alpha_{j}-\alpha_{i})}}{\left|\sin{(\alpha_{j}-\alpha_{i})}\right|^{3}}=0,
    \end{align*}
    which proves that (\ref{Crit3C}) is met. This completes the proof.
  \section{Proof of Theorem~\ref{Main Theorem 1a}}\label{Section proof of main theorem 1a}
    Let $V$ be the set of all $j\in\{1,\ldots,n\}$ for which there is an $i\in\{1,\ldots,n\}$ such that $\cos{(\alpha_{j}-\alpha_{i})}\neq\cos{(\beta_{j}-\beta_{i})}$. If $V$ is not empty, then relabeling the point masses if necessary, we may assume that $V=\{1,\ldots,k\}$, $k\leq n$. Additionally, again relabeling the point masses if necessary, we may assume that $\alpha_{1}<\alpha_{2}<\ldots<\alpha_{k}$, 
    where we define $\alpha_{j+kp}=\alpha_{j}+2\pi p$, $p\in\mathbb{Z}$. To prove that the $(q_{j1},q_{j2})^{T}$, $j\in V$ represent vertices of a regular polygon, it turns out we cannot copy the exact approach used in the proof of Theorem~\ref{Main Theorem 1}. While the idea we will use in this proof is similar, it requires a bit of tweaking:
    Let $\widehat{i}\in V$. We have by Lemma~\ref{Lemma 2} that for all $\widehat{j}\in\{1,\ldots,k\}$, $\widehat{j}\neq\widehat{i}$, there might be one term
    \begin{align*}
      \frac{m_{j}\sin{(\alpha_{j}-\alpha_{\widehat{i}})}}{\left(1-\left(\cos{(\beta_{j}-\beta_{\widehat{i}})}+r^{2}\left(\cos{(\alpha_{j}-\alpha_{\widehat{i}})}-\cos{(\beta_{j}-\beta_{\widehat{i}})}\right)\right)^{2}\right)^{\frac{3}{2}}}
    \end{align*}
    in (\ref{Crit1}) that can cancel out against
    \begin{align*}
      \frac{m_{\widehat{j}}\sin{(\alpha_{\widehat{j}}-\alpha_{\widehat{i}})}}{\left(1-\left(\cos{(\beta_{\widehat{j}}-\beta_{\widehat{i}})}+r^{2}\left(\cos{(\alpha_{\widehat{j}}-\alpha_{\widehat{i}})}-\cos{(\beta_{\widehat{j}}-\beta_{\widehat{i}})}\right)\right)^{2}\right)^{\frac{3}{2}}}
    \end{align*}
    for which $\cos{(\alpha_{j}-\alpha_{\widehat{i}})}=\cos{(\alpha_{\widehat{j}}-\alpha_{\widehat{i}})}$. We will refer to that $j$ as $j_{1}$. Additionally, again by Lemma~\ref{Lemma 2}, there might be such terms for which $\cos{(\alpha_{j}-\alpha_{\widehat{i}})}=-\cos{(\alpha_{\widehat{j}}-\alpha_{\widehat{i}})}$, in which case we, again by Lemma~\ref{Lemma 2}, also have that $\cos{(\beta_{j}-\beta_{\widehat{i}})}=-\cos{(\beta_{\widehat{j}}-\beta_{\widehat{i}})}$. There are at most two such $j$s and we will refer to them as $j_{2}$ and $j_{3}$. Note that $\sin{(\alpha_{\widehat{j}}-\alpha_{\widehat{i}})}\neq 0$: If $\sin{(\alpha_{\widehat{j}}-\alpha_{\widehat{i}})}=0$, then $\alpha_{\widehat{j}}=\alpha_{\widehat{i}}$, or $\alpha_{\widehat{j}}=\alpha_{\widehat{i}}\pm\pi$. If $\alpha_{\widehat{j}}=\alpha_{\widehat{i}}$, then we have a contradiction, as by construction the $\alpha$s are distinct. If $\alpha_{\widehat{j}}=\alpha_{\widehat{i}}\pm\pi$, then by Lemma~\ref{Lemma 2} and (\ref{Crit2}) we have that $\sin{(\beta_{\widehat{j}}-\beta_{\widehat{i}})}=0$ as well, which means that as $\beta_{\widehat{j}}\neq\beta_{\widehat{i}}$ for the same reason as why $\alpha_{\widehat{j}}\neq\alpha_{\widehat{i}}$ that $\beta_{\widehat{j}}=\beta_{\widehat{i}}\pm\pi$, which means that $\langle q_{\widehat{i}},q_{\widehat{j}}\rangle=-1$, which makes the right-hand side of (\ref{EquationsOfMotion Curved}) undefined for $i=\widehat{i}$. This also means that $\sin{(\beta_{\widehat{j}}-\beta_{\widehat{i}})}\neq 0$.
    If only $j_{1}$ exists, then by Lemma~\ref{Lemma 2} and (\ref{Crit1}) we have as $\alpha_{\widehat{j}}$ and $\alpha_{j_{1}}$ are distinct that $m_{\widehat{j}}=m_{j_{1}}$. If only either $j_{2}$, or $j_{3}$ exist, then we may assume that only $j_{2}$ exists and because $\langle q_{\widehat{j}},q_{j_{2}}\rangle=-1$ would mean that the right-hand side of (\ref{EquationsOfMotion Curved}) is undefined for $i=\widehat{j}$, we have by Lemma~\ref{Lemma 2} and (\ref{Crit1}), or (\ref{Crit2}) that $0=m_{\widehat{j}}+m_{j_{2}}>0$, which is a contradiction. If $j_{1}$ and $j_{2}$ exist, then because the $\alpha_{i}$ are distinct for all $i\in\{1,\ldots,n\}$, the $\beta_{i}$ are distinct for all $i\in\{1,\ldots,n\}$, $\langle q_{\widehat{j}},q_{j_{2}}\rangle\neq -1$ and $\langle q_{j_{1}},q_{j_{2}}\rangle\neq -1$, we have by (\ref{Crit1}), (\ref{Crit2}) and Lemma~\ref{Lemma 2} that
    \begin{align*}
      0=m_{\widehat{j}}-m_{j_{1}}+m_{j_{2}}\quad\textrm{and}\quad 0=m_{\widehat{j}}-m_{j_{1}}-m_{j_{2}},
    \end{align*}
    which means that $0=m_{j_{2}}>0$, which is a contradiction. If $j_{1}$ does not exist, but $j_{2}$ and $j_{3}$ do, then by (\ref{Crit1}), (\ref{Crit2}), Lemma~\ref{Lemma 2}, $\langle q_{\widehat{j}},q_{j_{2}}\rangle\neq -1$ and $\langle q_{\widehat{j}},q_{j_{3}}\rangle\neq -1$ we have that
    \begin{align*}
      0=m_{\widehat{j}}+m_{j_{2}}-m_{j_{3}}\quad\textrm{and}\quad
      0=m_{\widehat{j}}-m_{j_{2}}+m_{j_{3}}.
    \end{align*}
    Adding these last two equations gives $0=2m_{\widehat{j}}>0$, which is a contradiction. The only possibility left is then that $j_{1}$, $j_{2}$ and $j_{3}$ all exist. In that case we get by (\ref{Crit1}), (\ref{Crit2}), Lemma~\ref{Lemma 2}, $\langle q_{\widehat{j}},q_{j_{2}}\rangle\neq -1$, $\langle q_{\widehat{j}},q_{j_{3}}\rangle\neq -1$, $\langle q_{j_{1}},q_{j_{2}}\rangle\neq -1$ and $\langle q_{j_{1}},q_{j_{3}}\rangle\neq -1$ that
    \begin{align*}
      0=m_{\widehat{j}}-m_{j_{1}}+m_{j_{2}}-m_{j_{3}}\quad\textrm{and}\quad 0=m_{\widehat{j}}-m_{j_{1}}-m_{j_{2}}+m_{j_{3}}.
    \end{align*}
    By these two equations we then get that $m_{\widehat{j}}=m_{j_{1}}$ and $m_{j_{2}}=m_{j_{3}}$.
    From this relatively long argument, we can conclude that $j_{1}$ has to exist, which means that in particular there is such a $j_{1}$ for $\widehat{j}=\widehat{i}+1$, which then means that $\alpha_{\widehat{i}+1}-\alpha_{\widehat{i}}=\alpha_{\widehat{i}}-\alpha_{\widehat{i}-1}$ for all $\widehat{i}\in V$. This proves that the $(q_{j1},q_{j2})^{T}$, $j\in V$ represent the vertices of a regular polygon with $|V|$ vertices, as do the $(q_{j3},q_{j4})^{T}$, $j\in V$. Note that we can now immediately conclude from subtracting (\ref{Crit3})for $i=\widehat{i}$ from (\ref{Crit3}) for $i=\widehat{i}+1$ that $m_{\widehat{i}}=m_{\widehat{i}+1}$, which shows that all masses $m_{j}$, $j\in V$ are equal.

    If $V\neq\{1,\ldots,n\}$, then for $\widehat{j}\notin V$ we have by construction if $|V|=k$, $s\in V$ that
    \begin{align*}
      \cos{\left(\alpha_{\widehat{j}}-\frac{2\pi s}{k}\right)}=\cos{\left(\beta_{\widehat{j}}-\frac{2\pi p_{s}}{k}\right)},
    \end{align*}
    where $p_{1},\textrm{ }\ldots,\textrm{ }p_{k}\in V$ define a permutation of the elements $1$, \ldots, $k$ such that $\cos{(\alpha_{j}-\alpha_{i})}\neq\cos{(\beta_{j}-\beta_{i})}$ for $j$, $i\in V$. So for any $\widehat{j}\notin V$ we have that
    \begin{align*}
      \alpha_{\widehat{j}}-\beta_{\widehat{j}}=\frac{2\pi (s-p_{s})}{k}\textrm{ }(\textrm{mod } 2\pi)\quad\textrm{or}\quad\alpha_{\widehat{j}}+\beta_{\widehat{j}}=\frac{2\pi (s+p_{s})}{k}\textrm{ }(\textrm{mod } 2\pi).
    \end{align*}
    Note that by construction $V$ has at least three elements, which means that there are at least two values of $s_{1}$, $s_{2}\in V$ for which
    \begin{align*}
      &\alpha_{\widehat{j}}-\beta_{\widehat{j}}=\frac{2\pi (s_{i}-p_{s_{i}})}{k}\textrm{ }(\textrm{mod } 2\pi),\quad\text{or}\quad\alpha_{\widehat{j}}+\beta_{\widehat{j}}=\frac{2\pi (s_{i}+p_{s_{i}})}{k}\textrm{ }(\textrm{mod } 2\pi),\quad i\in\{1,2\}.
    \end{align*}
    Therefore
    \begin{align*}
      \frac{2\pi (s_{1}-p_{s_{1}})}{k}=\frac{2\pi (s_{2}-p_{s_{2}})}{k}\textrm{ }(\textrm{mod } 2\pi),\quad\text{or}\quad\frac{2\pi (s_{1}+p_{s_{1}})}{k}=\frac{2\pi (s_{2}+p_{s_{2}})}{k}\textrm{ }(\textrm{mod } 2\pi).
    \end{align*}
    This gives
    \begin{align*}
      \cos{\frac{2\pi (s_{1}-s_{2})}{k}}=\cos{\frac{2\pi (p_{s_{1}}-p_{s_{2}})}{k}},
    \end{align*}
    so $\cos{\left(\alpha_{s_{1}}-\alpha_{s_{2}}\right)}=\cos{\left(\beta_{s_{1}}-\beta_{s_{2}}\right)}$, while $s_{1}$, $s_{2}\in V$. This is a contradiction. This then finally proves that the $(q_{j1},q_{j2})^{T}$, $j\in\{1,\ldots,n\}$ are the vertices of a regular polygon with $n$ vertices, as are the $(q_{j3},q_{j4})^{T}$, $j\in\{1,\ldots,n\}$ and that all masses are equal, or that if $V=\emptyset$ by Theorem~\ref{Main Theorem 2} the $q_{j}$, $j\in\{1,\ldots,n\}$ are the vertices of a (possibly irregular) polygon inscribed in a great circle. This completes the proof.
  \section{Proof of Theorem~\ref{Main Theorem 3}}\label{Section proof of main theorem 3}
    We will set out to create a solution of the desired type that solves (\ref{Crit1})--(\ref{Crit3}): Let $j\in\{1,\ldots,n\}$, \begin{align*}W_{j}=\{s\in\{1,\ldots,n\}\mid\alpha_{s}=\alpha_{j}\},\quad V_{j}=\{s\in\{1,\ldots,n\}\mid\beta_{s}=\beta_{j}\}\end{align*} and select $j_{1}$,\textrm{ }\ldots,\textrm{ }$j_{k}\in\{1,\ldots,n\}$ such that $W_{s_{1}}\cap W_{s_{2}}=\emptyset$ for $s_{1}$, $s_{2}\in\{j_{1},\ldots,j_{k}\}$, $s_{1}\neq s_{2}$ and $\bigcup\limits_{u=1}^{k}W_{j_{u}}=\{1,\ldots,n\}$ and $\widehat{j}_{1}$, \ldots, $\widehat{j}_{p}\in\{1,\ldots,n\}$ such that $V_{s_{1}}\cap V_{s_{2}}=\emptyset$ for $s_{1}$, $s_{2}\in\left\{\widehat{j}_{1},\ldots,\widehat{j}_{p}\right\}$, $s_{1}\neq s_{2}$ and $\bigcup\limits_{u=1}^{p}V_{\widehat{j}_{u}}=\{1,\ldots,n\}$. Additionally, as all masses are equal, we will write $m_{j}=m$. Let $i\in W_{j_{a}}$, $a\in\{1,\ldots,k\}$. Then (\ref{Crit3}) gives
    \begin{align}
      &\delta=r\rho^{2}\sum\limits_{\substack{l=1\\l\neq a}}^{k}\sum\limits_{j\in W_{j_{l}}}\frac{m(\cos{(\alpha_{j}-\alpha_{i})}-\cos{(\beta_{j}-\beta_{i})})}{\left(1-\left(\cos{(\beta_{j}-\beta_{i})}+r^{2}\left(\cos{(\alpha_{j}-\alpha_{i})}-\cos{(\beta_{j}-\beta_{i})}\right)\right)^{2}\right)^{\frac{3}{2}}}\nonumber\\
      &+r\rho^{2}\sum\limits_{j\in W_{j_{a}},\textrm{ }j\neq i}\frac{m(\cos{(\alpha_{j}-\alpha_{i})}-\cos{(\beta_{j}-\beta_{i})})}{\left(1-\left(\cos{(\beta_{j}-\beta_{i})}+r^{2}\left(\cos{(\alpha_{j}-\alpha_{i})}-\cos{(\beta_{j}-\beta_{i})}\right)\right)^{2}\right)^{\frac{3}{2}}}. \label{SummingIt1}
    \end{align}
    Let the number of elements in $W_{j_{l}}$ be $p_{1}$ and the number of elements in $V_{\widehat{j}_{u}}$ be $p_{2}$. Then because for $j\in W_{j_{a}}$ we have that $l=a$ and that the $(q_{j3},q_{j4})^{T}$ represent vertices of a regular polygon and because for $j\in V_{\widehat{j}_{u}}$ we have that the $(q_{j1},q_{j2})^{T}$ represent vertices of a regular polygon, (\ref{SummingIt1}) can be rewritten as
    \begin{align*}
      &\delta=mr\rho^{2}\sum\limits_{\substack{l=1\\l\neq a}}^{p_{2}}\sum\limits_{j=1}^{p_{1}}\frac{\cos{\frac{2\pi}{p_{2}}(l-a)}-\cos{\frac{2\pi}{p_{1}}(j-i)}}{\left(1-\left(\cos{\frac{2\pi}{p_{1}}(j-i)}+r^{2}\left(\cos{\frac{2\pi}{p_{2}}(l-a)}-\cos{\frac{2\pi}{p_{1}}(j-i)}\right)\right)^{2}\right)^{\frac{3}{2}}}\\
      &+mr\rho^{2}\sum\limits_{\substack{j=1\\j\neq i}}^{p_{1}}\frac{1-\cos{\frac{2\pi}{p_{1}}(j-i)}}{\left(1-\left(\cos{\frac{2\pi}{p_{2}}(j-i)}+r^{2}\left(1-\cos{\frac{2\pi}{p_{1}}(j-i)}\right)\right)^{2}\right)^{\frac{3}{2}}},
    \end{align*}
    which in turn can be rewritten as
    \begin{align}
      &\delta=mr\rho^{2}\sum\limits_{\substack{l=1\\l\neq a}}^{p_{2}}\sum\limits_{\substack{j=1\\j\neq i}}^{p_{1}}\frac{\cos{\frac{2\pi}{p_{2}}(l-a)}-\cos{\frac{2\pi}{p_{1}}(j-i)}}{\left(1-\left(\cos{\frac{2\pi}{p_{1}}(j-i)}+r^{2}\left(\cos{\frac{2\pi}{p_{2}}(l-a)}-\cos{\frac{2\pi}{p_{1}}(j-i)}\right)\right)^{2}\right)^{\frac{3}{2}}}\nonumber\\
      &+mr\rho^{2}\sum\limits_{\substack{l=1\\l\neq a}}^{p_{2}}\frac{\cos{\frac{2\pi}{p_{2}}(l-a)}-1}{\left(1-\left(1+r^{2}\left(\cos{\frac{2\pi}{p_{2}}(l-a)}-1\right)\right)^{2}\right)^{\frac{3}{2}}}\nonumber\\
      &+mr\rho^{2}\sum\limits_{\substack{j=1\\j\neq i}}^{p_{1}}\frac{1-\cos{\frac{2\pi}{p_{1}}(j-i)}}{\left(1-\left(\cos{\frac{2\pi}{p_{2}}(j-i)}+r^{2}\left(1-\cos{\frac{2\pi}{p_{1}}(j-i)}\right)\right)^{2}\right)^{\frac{3}{2}}}.  \label{SummingIt2}
    \end{align}
    Rewriting (\ref{SummingIt2}) in terms of $\widehat{l}=l-a$ and $\widehat{j}=j-i$ and then replacing $\widehat{l}$ and $\widehat{j}$ with $l$ and $j$ respectively gives
    \begin{align}
      &\delta=mr\rho^{2}\sum\limits_{l=1}^{p_{2}-1}\sum\limits_{j=1}^{p_{1}-1}\frac{\cos{\frac{2\pi l}{p_{2}}}-\cos{\frac{2\pi j}{p_{1}}}}{\left(1-\left(\cos{\frac{2\pi j}{p_{1}}}+r^{2}\left(\cos{\frac{2\pi l}{p_{2}}}-\cos{\frac{2\pi j}{p_{1}}}\right)^{2}\right)\right)^{\frac{3}{2}}}\nonumber\\
      &+mr\rho^{2}\sum\limits_{l=1}^{p_{2}-1}\frac{\cos{\frac{2\pi l}{p_{2}}}-1}{\left(1-\left(1+r^{2}\left(\cos{\frac{2\pi l}{p_{2}}}-1\right)\right)^{2}\right)^{\frac{3}{2}}}\nonumber\\
      &+mr\rho^{2}\sum\limits_{j=1}^{p_{1}-1}\frac{1-\cos{\frac{2\pi j}{p_{1}}}}{\left(1-\left(\cos{\frac{2\pi j}{p_{2}}}+r^{2}\left(1-\cos{\frac{2\pi j}{p_{1}}}\right)\right)^{2}\right)^{\frac{3}{2}}}.  \label{SummingIt3}
    \end{align}
    Note that (\ref{SummingIt3}) is independent of $i$, which means that (\ref{Crit3}) of Criterion~\ref{Criterion} is met, as long as we can find an $r$ that solves the differential equation posed by (\ref{Crit3}):
    By Lemma~\ref{Lemma 1} we have that $\dot{\theta}=\frac{C_{1}}{r^{2}}$ and $\dot{\phi}=\frac{C_{2}}{\rho^{2}}$ for suitable constants $C_{1}$ and $C_{2}$, which means that
    \begin{align}\label{EnergySequel1}
      \frac{2\dot{r}}{\rho^{2}}\delta &=\frac{2\dot{r}}{\rho^{2}}\left(\ddot{r}+r\rho^{2}\left(\frac{C_{2}^{2}}{\rho^{4}}-\frac{C_{1}^{2}}{r^{4}}\right)+r\left(\frac{\dot{r}}{\rho}\right)^{2}\right)\nonumber\\
      &=\frac{d}{dt}\left((\dot{r})^{2}+(\dot{\rho})^{2}+\frac{C_{1}^{2}}{r^{2}}+\frac{C_{2}^{2}}{\rho^{2}}\right).
    \end{align}
    Additionally, note that
    \begin{align}\label{EnergySequel2}
      &\frac{2\dot{r}}{\rho^{2}}\cdot r\rho^{2}\sum\limits_{\substack{j=1\\j\neq i}}^{n}\frac{m_{j}(\cos{(\alpha_{j}-\alpha_{i})}-\cos{(\beta_{j}-\beta_{i})})}{\left(1-\left(\cos{(\beta_{j}-\beta_{i})}+r^{2}\left(\cos{(\alpha_{j}-\alpha_{i})}-\cos{(\beta_{j}-\beta_{i})}\right)\right)^{2}\right)^{\frac{3}{2}}}\nonumber\\
      &=\frac{d}{dt}\sum\limits_{\substack{j=1\\j\neq i}}^{n}\frac{m_{j}\left(\cos{(\beta_{j}-\beta_{i})}+r^{2}\left(\cos{(\alpha_{j}-\alpha_{i})}-\cos{(\beta_{j}-\beta_{i})}\right)\right)}{\left(1 -\left(\cos{(\beta_{j}-\beta_{i})}+r^{2}\left(\cos{(\alpha_{j}-\alpha_{i})}-\cos{(\beta_{j}-\beta_{i})}\right)\right)^{2}\right)^{\frac{1}{2}}}.
    \end{align}
    Inserting (\ref{SummingIt3}) into (\ref{EnergySequel1}), using (\ref{EnergySequel2}) and integrating the resulting equation now gives
    \begin{align}
      &(\dot{r})^{2}+(\dot{\rho})^{2}+\frac{C_{1}^{2}}{r^{2}}+\frac{C_{2}^{2}}{\rho^{2}}\nonumber\\
      &=\frac{1}{2}m\sum\limits_{l=1}^{p_{2}-1}\sum\limits_{j=1}^{p_{1}-1}\frac{\cos{\frac{2\pi j}{p_{1}}}+r^{2}\left(\cos{\frac{2\pi l}{p_{2}}}-\cos{\frac{2\pi j}{p_{1}}}\right)}{\left(1-\left(\cos{\frac{2\pi j}{p_{1}}}+r^{2}\left(\cos{\frac{2\pi l}{p_{2}}}-\cos{\frac{2\pi j}{p_{1}}}\right)^{2}\right)\right)^{\frac{1}{2}}}\nonumber\\
      &+\frac{1}{2}m\sum\limits_{l=1}^{p_{2}-1}\frac{1+r^{2}\left(\cos{\frac{2\pi l}{p_{2}}}-1\right)}{\left(1-\left(1+r^{2}\left(\cos{\frac{2\pi l}{p_{2}}}-1\right)\right)^{2}\right)^{\frac{1}{2}}}\nonumber\\
      &+\frac{1}{2}m\sum\limits_{j=1}^{p_{1}-1}\frac{\cos{\frac{2\pi j}{p_{2}}}+r^{2}\left(1-\cos{\frac{2\pi j}{p_{1}}}\right)}{\left(1-\left(\cos{\frac{2\pi j}{p_{2}}}+r^{2}\left(1-\cos{\frac{2\pi j}{p_{1}}}\right)\right)^{2}\right)^{\frac{1}{2}}}+C \label{SummingItR2}.
    \end{align}
    for a suitable constant $C$.
    We now have that (\ref{SummingItR2}) is a first-order ordinary differential equation with $r$ as its solution. For suitable $m$, $C$, $C_{1}$ and $C_{2}$ one can find solutions of this equation where $r$ is a constant function. Choosing an initial value of $r$ between $0$ and $1$ that is not equal to one of those constant solutions gives a nonconstant solution. This proves that (\ref{Crit3}) is satisfied and gives rise to both constant and nonconstant solutions $r$.

    Repeating the calculation that lead to (\ref{SummingIt3}), but starting with the righthand side of (\ref{Crit1}) instead gives
    \begin{align}
      &0=mr\rho^{2}\sum\limits_{l=1}^{p_{2}-1}\sum\limits_{j=1}^{p_{1}-1}\frac{\sin{\frac{2\pi l}{p_{2}}}}{\left(1-\left(\cos{\frac{2\pi j}{p_{1}}}+r^{2}\left(\cos{\frac{2\pi l}{p_{2}}}-\cos{\frac{2\pi j}{p_{1}}}\right)\right)^{2}\right)^{\frac{3}{2}}}\nonumber\\
      &+mr\rho^{2}\sum\limits_{l=1}^{p_{2}-1}\frac{\sin{\frac{2\pi l}{p_{2}}}}{\left(1-\left(1+r^{2}\left(\cos{\frac{2\pi l}{p_{2}}}-1\right)\right)^{2}\right)^{\frac{3}{2}}}. \label{SummingIt3SineAlpha}
    \end{align}
    Note that
    \begin{align*}
      &\sum\limits_{l=1}^{p_{2}-1}\frac{\sin{\frac{2\pi l}{p_{2}}}}{\left(1-\left(\cos{\frac{2\pi j}{p_{1}}}+r^{2}\left(\cos{\frac{2\pi l}{p_{2}}}-\cos{\frac{2\pi j}{p_{1}}}\right)^{2}\right)\right)^{\frac{3}{2}}}\\
      &=\sum\limits_{l=1}^{p_{2}-1}\frac{\sin{\frac{2\pi(p_{2}-l)}{p_{2}}}}{\left(1-\left(\cos{\frac{2\pi j}{p_{1}}}+r^{2}\left(\cos{\frac{2\pi(p_{2}-l)}{p_{2}}}-\cos{\frac{2\pi j}{p_{1}}}\right)^{2}\right)\right)^{\frac{3}{2}}},
    \end{align*}
    so
    \begin{align}
      \sum\limits_{l=1}^{p_{2}-1}\frac{\sin{\frac{2\pi l}{p_{2}}}}{\left(1-\left(\cos{\frac{2\pi j}{p_{1}}}+r^{2}\left(\cos{\frac{2\pi l}{p_{2}}}-\cos{\frac{2\pi j}{p_{1}}}\right)^{2}\right)\right)^{\frac{3}{2}}}=0.\label{SineHelp1}
    \end{align}
    By the same argument, we also have that
    \begin{align}
      \sum\limits_{l=1}^{p_{2}-1}\frac{\sin{\frac{2\pi l}{p_{2}}}}{\left(1-\left(1+r^{2}\left(\cos{\frac{2\pi l}{p_{2}}}-1\right)\right)^{2}\right)^{\frac{3}{2}}}=0.\label{SineHelp2}
    \end{align}
    So by (\ref{SummingIt3SineAlpha}), (\ref{SineHelp1}) and (\ref{SineHelp2}), we now have that (\ref{Crit1}) of Criterion~\ref{Criterion} is met as well.

    Switching the roles of the alphas and the betas in this last calculation then finally gives that (\ref{Crit2}) of Criterion~\ref{Criterion} is also met. This proves that positive elliptic-elliptic rotopulsator solutions $q_{1}$, \ldots, $q_{n}$ for which $r_{i}$ and $\rho_{i}$ are independent of $i$, for which the $(q_{i1},q_{i2})^{T}$, $i\in\{1,\ldots,n\}$ and the $(q_{i3},q_{i4})^{T}$, $i\in\{1,\ldots,n\}$ represent vertices of a regular polygon exist if the masses are all equal. This completes the proof.
    \begin{remark}
      It should be mentioned that a generalised result of (\ref{SummingItR2}) for $r_{i}$ and $\rho_{i}$ not necessarily independent of $i$ was obtained in \cite{DK}. However, we needed to explicitly deduce (\ref{SummingItR2}) here to prove that (\ref{Crit3}) of Criterion~\ref{Criterion} was met.
    \end{remark}
\bibliographystyle{amsplain}

\end{document}